\documentclass[numbook,smallextended]{svjour3}
\smartqed

\usepackage{epsfig}
\usepackage{amsmath}
\usepackage{amssymb}
\usepackage{amscd}
\usepackage{graphicx}
\usepackage{color}
\usepackage{verbatim}
\usepackage{dsfont}

\spnewtheorem*{xproof}{}{\itshape}{\rmfamily}% the label is assigned later

\renewenvironment{proof}[1][\proofname]
 {\renewcommand\xproofname{#1}\xproof}
 {\endxproof}

\newtheorem{thm}{Theorem}[section]
\newtheorem{lem}[thm]{Lemma}

\newtheorem{fact}[thm]{Fact}

\newtheorem{cor}[thm]{Corollary}
\newtheorem{defn}[thm]{Definition}

\spnewtheorem{rem}[thm]{Remark}{\it}{\rm}
\spnewtheorem{exam}[thm]{Example}{\it}{\rm}

\def \N {\mathbb N}

\def \Z {\mathbb Z}
\def \R {\mathcal R}
\def \F {\mathcal F}

\def \M {\mathcal M}
\def \P {\mathcal P}

\def \id {\mathsf{Id}}

\def \sq {sequence}
\def \xt {$(X,T)$}
\def \ys {$(Y,S)$}

\def \yst {$(\hat Y,\hat S)$}
\def \mtx {\mathcal M_T(X)}
\def \mtxp {\mathcal M_{T'}(X')}
\def \mtxt {\mathcal M_{\hat T}(\hat X)}
\def \msy {\mathcal M_S(Y)}
\def \mtsn {\mathcal M_{\hat T}(\mathbf S_n)}

\def \xtp {$(X',T')$}
\def \xtt {$(\hat X,\hat T)$}
\def \xttp {$(\hat X',\hat T')$}
\def \xtu {$(\mathfrak X,\mathfrak T)$}

\def \xtup {$(\mathfrak X',\mathfrak T')$}

\def \tl {topological}
\def \im {invariant measure}
\def \inv {invariant}
\def \ds {dynamical system}
\def \usc {upper semicontinuous}

\def \htop {h_{\mathsf{top}}}
\def \hsex {h_{\mathsf{sex}}}
\def \hemb {h_{\mathsf{emb}}}
\def \diam {\mathsf{diam}}
\def \eh {E_{\mathsf{min}}}

\def \Per {\mathsf{Per}}
\def \psup {\mathfrak P_{\mathsf{sup}}}
\def \plim {\mathfrak P_{\mathsf{lim}}}
\def \pmax {p^{\mathsf{max}}}
\def \pmin {p^{\mathsf{min}}}

\numberwithin{equation}{section}

\begin{document}

\title{Uniform generators, symbolic extensions with an embedding, and structure of periodic orbits \thanks{The research of the second author is supported by the NCN (National Science Center, Poland) grant 2013/08/A/ST1/00275.}}
\titlerunning{Uniform generators}
\author{David Burguet and Tomasz Downarowicz}

\institute{David Burguet
\at CNRS-UPMC Univ Paris 06, UMR 7599, Laboratoire de Probabilites et Modeles Aleatoires, 4 place Jussieu,
F-75005 Paris,\\\email{david.burguet@upmc.fr}
\and
Tomasz Downarowicz 
\at Faculty of Pure and Applied Mathematics, Wroc{\l}aw University of Technology, Wroc{\l}aw 50-370, Poland,\\\email{Tomasz.Downarowicz@pwr.edu.pl }
}

\date{Received: date / Accepted: date}

\maketitle

\begin{abstract}For a \tl\ \ds\ \xt\ we define a uniform generator as a finite measurable partition such that the symmetric cylinder sets in the generated process shrink in diameter uniformly to zero. The problem of existence and optimal cardinality of uniform generators has lead us to new challenges in \emph{the theory of symbolic extensions}. At the beginning we show that uniform generators can be identified with so-called symbolic extensions with an embedding, i.e., symbolic extensions admitting an equivariant measurable selector from preimages. From here we focus on such extensions and we strive to characterize the collection of the corresponding extension entropy functions on \im s. For aperiodic zero-dimensional systems we show that this collection coincides with that of extension entropy functions in arbitrary symbolic extensions, which, by the general theory of symbolic extensions, is known to coincide with the collection of all affine superenvelopes of the entropy structure of the system. In particular, we recover, after \cite{Bu16}, that an aperiodic zero-dimensional system is asymptotically $h$-expansive if and only if it admits an isomorphic symbolic extension. Next we pass to systems with periodic points, and we introduce the notion of a period tail structure, which captures the local growth rate of periodic orbits. Finally, we succeed in precisely identifying the wanted collection of extension entropy functions in symbolic extensions with an embedding: these are all the affine superenvelopes of the usual entropy structure which lie above certain threshold function determined by the period tail structure. This characterization allows us, among other things, to give estimates (and in examples to compute precisely) the optimal cardinality of a uniform generator. As a byproduct, we prove a theorem saying that every zero-dimensional system admits an aperiodic zero-dimensional extension which is isomorphic on aperiodic measures and otherwise principal (periodic measures lift to measures of entropy zero).
\end{abstract}

\section{Introduction}

To make the subject of this paper precise, we start the introduction with a formal definition:

\begin{defn}\label{pergen}
By a \emph{uniform generator} in an invertible \tl\ \ds\ \xt, where $T$ is a homeomorphism of a compact metric space $X$, we mean a finite measurable\footnote{See Remark \ref{sm} for the precise meaning of measurability.} partition $\P$ of $X$ satisfying 
$$
\lim_n\diam(\P^{[-n,n]}) = 0,
$$
where $\P^{[-n,n]}=\bigvee_{i=-n}^nT^i\P$ and $\diam(\P)$ is the maximal diameter of atoms of $\P$.
\end{defn}

Notice that since the partitions $\P^{[-n,n]}$ separate points, $\P$ is a Krieger generator simultaneously for all \im s on $X$. However, not every such simultaneous generator is uniform; for that the distance of separation must shrink uniformly throughout the space (see also Remark \ref{Hoch}). For example, if \ys\ is a subshift and $\P_\Lambda$ is the \emph{zero-coordinate partition} (whose atoms are cylinder sets over one-blocks corresponding to the symbols in the alphabet $\Lambda$), then $\P_\Lambda$ is a uniform generator (it is also clopen, which makes it specifically good). Uniform (not necessarily clopen) generators exist in some non-expansive systems as well: take for instance the partition into any two complementary arcs in an irrational rotation of the circle.

In this paper we focus on the existence and optimal cardinality of uniform generators in \tl\ \ds s,
a task which turns out unexpectedly intricate, and leading to new developments in the entropy theory of symbolic extensions. At the beginning of our study we make a crucial observation which links uniform generators with symbolic extensions. 
 
\begin{thm}\label{sg}
Let \xt\ be a \tl\ \ds. The following conditions are equivalent:
\begin{enumerate}
	\item There exists a uniform generator\footnote{Later we will show that the measurability 
	assumption of $\P$ can be dropped. That is, the existence of a ``non-measurable uniform 
	generator'' implies the existence of a measurable one---see Remark \ref{measu}.} $\P$ in \xt, of 
	cardinality $\ell$;
	\item There exists a symbolic extension $\pi:(Y,S)\to(X,T)$ over an alphabet of cardinality $\ell$, 
	which admits an \emph{equivariant measurable selector from preimages}, i.e., a map $\psi:X\to Y$ such
	that 
	\begin{enumerate}
	\item $\psi\circ T = S\circ\psi$, and
	\item $\pi\circ\psi = \id_X$.
	\end{enumerate}
\end{enumerate}	
Terminology: Since the map $\psi$ is a measurable embedding of \xt\ into \ys, such a $\pi$ (or the system \ys) is called a \emph{symbolic extension with an embedding}. 

\end{thm}

\begin{proof}
Let $\P$ be a uniform generator and let $\Lambda$ be an alphabet which bijectively labels the atoms of $\P$, i.e., $\P=\{P_a:a\in\Lambda\}$. The map $\psi:X\to \Lambda^\Z$ assigning to each $x$ its bilateral $\P$-name (by the rule $(\psi(x))(n) = a\ \iff\ T^n(x)\in P_a$) is measurable and equivariant. Let $Y=\overline{\psi(X)}$. Clearly, $Y$ is a subshift. For $y\in Y$ and $n\in\N$ let $C_n(y)$ denote the set of points whose $\P$-name coincides on the interval $[-n,n]$ with the central block $y[-n,n]$ of length $2n+1$ of $y$. Clearly, this set is an atom of $\P^{[-n,n]}$, so its diameter does not exceed $\diam(\P^{[-n,n]})$, and it is nonempty because the same block must occur in $\psi(x)$ for some $x\in X$. Thus the intersection $\bigcap_n \overline{C_n(y)}$ contains exactly one point which we denote as $\pi(y)$. It is obvious that $\pi:Y\to X$ is a \tl\ factor map and $\psi$ is a selector from its preimages.

Now assume that \xt\ has a symbolic extension $\pi:(Y,S)\to(X,T)$ admitting a required selector $\psi$. Let $\P_\Lambda$ denote the zero-coordinate partition of $Y$ and define $\P=\psi^{-1}(\P_\Lambda)$. Clearly, $\P$ is a measurable partition of $X$ of cardinality $\ell$, the same as that of the alphabet $\Lambda$ of $Y$. The convergence of the diameters of $\P^{[n,n]}$ to zero follows directly from the three facts: that the same property has $\P_\Lambda$ in $Y$, that each atom of $\P^{[-n,n]}$ is contained in the image by $\pi$ of an atom of $\P_\Lambda^{[-n,n]}$, and that $\pi$ is uniformly continuous. 
\qed \end{proof}

The above theorem allows us to switch from searching, among measurable partitions, for a uniform generator to studying symbolic extensions (with additional properties), which is a fairly well understood field. In this setup the main object of our interest will be the collection of extension entropy functions in symbolic extensions with an embedding. Once we manage to describe this collection, we can compute the optimal cardinality of a uniform generator very easily. Our main results are formulated for zero-dimensional systems, however, they equally apply to systems \xt\ which admit an \emph{isomorphic zero-dimensional extension}, i.e., a topological zero-dimensional extension such that the corresponding factor map is an isomorphism between each \im\ and its \emph{unique} preimage. The class of systems admitting an isomorphic zero-dimensional extension includes those which have the small boundary property, which is a large and well described class (however, it is unknown whether the two conditions are equivalent).

We prove that if \xt\ is zero-dimensional aperiodic (contains no periodic orbits) then the collection of extension entropy functions in symbolic extensions with an embedding is the same as that for general symbolic extensions. If, in addition, \xt\ is asymptotically $h$-expansive then we recover the result first obtained in \cite{Bu16}\footnote{The cited paper is the first work dealing with the subject of symbolic extensions with an embedding. For asymptotically $h$-expansive systems the existence of an isomorphic extension is proved even in presence of periodic points as long as they satisfy a condition called \emph{asymptotic per-expansiveness}.}, that it admits an \emph{isomorphic symbolic extension}, i.e., such that $\psi(X)$ has full measure for all \im s on \ys. Because the existence of an isomorphic symbolic extension implies (by former results) asymptotic $h$-expansiveness, for aperiodic zero-dimensional systems we obtain an equivalence.

The most interesting phenomena occur when \xt\ does have periodic points. The growth rate of periodic orbits then may influence the entropy of symbolic extensions with an embedding (and thus the cardinality of a uniform generator). The simplest example is a system \xt\ consisting exclusively of periodic points. Every such system admits a symbolic extension of entropy zero. But a symbolic extension with an embedding must have at least as many periodic orbits of every period as \xt, which, in view of the simple fact that the \tl\ entropy of a symbolic system is always at least as large as the exponential growth rate of periodic orbits, implies positive entropy of the extension, if there are sufficiently many periodic orbits in \xt\ (see also Example \ref{example1} in Section \ref{sec4}). So, for systems with periodic points, not all extension entropy functions appearing in symbolic extensions occur in symbolic extensions with an embedding, and a new challenge is to say precisely which ones do. We succeed in providing a characterization by introducing a \sq\ of functions defined on \im s, called the \emph{period tail structure}. This \sq\ is in some sense analogous to the entropy structure, but depends exclusively on the distribution of periodic orbits in \xt. By combining the period tail structure with the usual entropy structure we manage to identify the collection of appropriate superenvelopes, i.e., of the desired extension entropy functions. 

\begin{rem}\label{sm}
Recall that whenever we pass from a \tl\ \ds\ to a measure-theoretic system by fixing one of its \im s $\mu$, in order to obtain a standard probability space we must consider the sigma-algebra of Borel sets \emph{completed} with respect to $\mu$. Notice that the intersection of all such sigma-algebras (i.e., ``universally measurable sets'') coincides with the completion of the Borel sigma-algebra with respect to the sigma-ideal of the \emph{null sets}, i.e., sets of measure zero for all \im s. From now on, by a \emph{measurable set} we will mean any member of this completion.
\end{rem}

\begin{rem}\label{me}
Condition (2) in Theorem \ref{sg} can be weakened: it suffices that the map $\psi$ is defined almost everywhere on $X$, i.e., except on a null set. Indeed, it is always possible to prolong the map to the missing null set. First of all, if necessary, we can enlarge the null set $A$ so that it becomes invariant, (i.e., a union of entire orbits). Next, $A$ contains a set $B$ selecting exactly one point from every orbit of $A$. Now, the mapping $\psi$ can first be defined on $B$ as an arbitrary selector from the preimages by $\pi$, and then prolonged to the rest of $A$ following the rules of equivariance. So defined $\psi|_A$ is measurable, because we operate within a null set whose all subsets are measurable by definition.
\end{rem}

\begin{rem} Later in this paper we will deal with symbolic extensions with what we call \emph{partial embedding}, i.e., a selector from preimages defined only on some measurable invariant subset $F\subset X$ (in the sequel this subset will support all ergodic measures except some periodic ones). 
It can be proved that the existence of a symbolic extension with partial embedding is equivalent to the existence of a \emph{uniform partial generator}, i.e., a partition $\P$ of $F$ which satisfies the condition of Definition \ref{pergen}. This approach will be useful in applications to smooth dynamics discussed in \cite{Bu16'} (see also the end of the last section).
\end{rem}

\section{Preliminaries}
\subsection{Symbolic extensions}
By a \emph{subshift} we mean a \ds\ \ys, where $Y\subset\Lambda^\Z$ is closed and shift-\inv, and $S$ 
is the shift map $(Sy)_n = y_{n+1}$ ($y=(y_n)_{n\in\Z}\in Y$). A \emph{symbolic extension} of
\xt\ is a subshift \ys\ together with a \tl\ factor map (i.e., continuous equivariant surjection) 
$\pi:(Y,S)\to (X,T)$. The map $\pi$ induces a continuous and affine surjection $\pi^*:\msy\to\mtx$ from shift-\inv\ probability measures on $Y$ to $T$-\inv\ probability measures on $X$, defined by $(\pi^*\nu)(B) = \nu(\pi^{-1}(B))$ ($B$ is a Borel set in $X$). We will skip the star and use $\pi$ for this induced map. With an extension $\pi$ we associate the \emph{extension entropy function} on $\mtx$ defined by the formula
$$
h^\pi(\mu) = \sup\{h_\nu(S): \nu\in\pi^{-1}(\mu)\},
$$
where $h_\nu(S)$ denotes the Kolmogorov-Sinai entropy of $\nu$ on $Y$.
The \emph{symbolic extension entropy function} is defined on $\mtx$ as
$$
\hsex(\mu) = \inf\{h^\pi(\mu): \pi\text{ is a symbolic extension of \xt}\}.
$$
Also, by $h$ we will denote the \emph{entropy function} on $\mtx$, $h(\mu) = h_\mu(T)$. Recall that
the set $\mtx$, when regarded with the weak-star topology, is a Choquet simplex (in particular, a compact convex set) and its extreme points are precisely the ergodic measures. Below we recall two basic facts concerning the symbolic extension entropy function.
\begin{itemize}
	\item Both $\hsex$ and $h^\pi$ (in any symbolic extension) are nonnegative upper semicontinuous 
	functions, and so are the differences $\hsex - h$ and $h^\pi - h$.
	\item (\emph{Symbolic extension entropy variational principle}): $\sup\{\hsex(\mu):\mu\in\mtx\}$ equals 
	$\hsex(X,T)$ defined as the infimum of $\htop(Y,S)$ over all symbolic extensions of \xt.
\end{itemize}
Further facts will be provided after entropy structure is introduced.

\subsection{Zero-dimensional systems}\label{zedysy}\hfill\break
\emph{Array representation}. Every \tl\ \ds\ \xt\ on a zero-di\-men\-sio\-nal space $X$ is conjugate to an inverse limit of subshifts. Practically, this means that it admits an \emph{array representation} in
which every point $x$ is an array $[x_{k,n}]_{k\ge 1,n\in\Z}$, where all entries $x_{k,n}$ belong a finite alphabet $\Lambda_k$ (which does not depend on $x$ or $n$). It can be arranged that all the alphabets $\Lambda_k$ are the same, or even equal to $\{0,1\}$. The map $T$ is the \emph{horizontal shift} $(Tx)_{k,n} = x_{k,n+1}$. By projection on the first $k$ rows, $\pi_k x = [x_{i,n}]_{i\in[1,k], n\in\Z}$ we obtain a \tl\ factor of \xt\, denoted by $(X_k,T_k)$, which is a subshift over the alphabet $\Delta_k = \prod_{i=1}^k\Lambda_i$. The projection $\pi_k$ naturally applies not only to $X$ but also to $X_{k+1}$, so that $(X_k,T_k)$ is a \tl\ factor of $(X_{k+1},T_{k+1})$. Then \xt\ is the inverse limit of the subshifts $(X_k,T_k)$.

\medskip\noindent
\emph{System of markers}. Let \xt\ be an aperiodic zero-dimensional system given in an array representation using some finite alphabets $\Lambda_k$. In order to allow inserting markers we need to enlarge the alphabets: in row $k$ we will be using $\Lambda^*_k=\Lambda_k\times\{\emptyset,|\}=\{a,\,a|: a\in\Lambda_k\}$. 
Now we recall the \emph{Krieger's marker lemma} (see \cite[Lemma 2.2]{Bo83}, here we use a version for zero-dimensional aperiodic systems): In any aperiodic zero-dimensional system \xt, for every $n\ge 1$ there exists a clopen set $F\subset Y$ such that:
\begin{enumerate}
	\item[(a)] $T^{i}(F)$ are pairwse disjoint for $i = 0,1,\dots,n-1$,
	\item[(b)] $\bigcup_{i=-n}^n T^{i}(F) =X$. 
\end{enumerate}
We choose a fast increasing \sq\ $\{n_k\}$ of natural numbers and by applying the above lemma with the parameters $n_k$ we obtain clopen marker sets $F_k\subset X$. We distribute ``preliminary'' markers in every row $k$ of every array $x\in X$ by the rule: if $T^i x\in F_k$ then we place a marker in $x$ at the position $(k,i)$. By a \emph{gap} between the neighboring markers, say at $(k,i)$ and $(k,j)$ we will always mean the interval $[i+1,j]$ in row $k$. The lengths of these gaps range between $n_k$ and $2n_k+1$. Because each $F_k$ is clopen, we obtain a conjugate representation of \xt\ with the markers. The last step will be called the \emph{upward adjustment}. Proceeding inductively (the first step being idle), for $k\ge 2$ we move every marker in row $k$ to align it with the nearest marker in row $k-1$, say, on the right. Notice that each marker is moved by at most $n_1+n_2+\dots+n_{k-1}$, which we can assume, is smaller than $\frac{n_k}2$. Since the new position of each marker depends on a bounded rectangle in the array, the algorithm is continuous, hence it produces a conjugate model of \xt. From now on by $x\in X$ we will understand the array with all the markers included. We can summarize the properties of the markers just introduced:
\begin{enumerate}
	\item[(A)] the gaps between two neighboring markers in row $k$ range between $\pmin_k=\frac{n_k}2$ and 
	$\pmax_k=\frac52n_k+1$,
	\item[(B)] for $k>1$ every marker in row $k$ is aligned with a marker in row $k-1$.
\end{enumerate}
The latter condition simply means that the marker sets $F_k$ (altered by the upward adjustment) are nested.
If we ignore the symbols from $\Lambda_k$ and keep only the markers, we will obtain an array system $(X_0,T_0)$ over two symbols $\{\emptyset,|\}$ which is a \tl\ factor of \xt, and which we will call \emph{the system of markers}. 

The block occupying a gap in row number $k$ of $x$, i.e, the positions between two neighboring markers, will be called a $k$-block (occurring in $x$). The rectangle extending vertically through the top $k$ rows and horizontally between two neighboring markers of the $k$th row of $x$, will be called a $k$-rectangle (occurring in $x$; see Figure \ref{rect}). 
\begin{figure}[ht]
\[
\begin{array}{c}
\hline
\ldots\,\vline\, 0\,1\,\vline\,1\,1\,\vline\,\mathbf{0\,0}\,\vline\,\mathbf{0\,1}\,\vline\,\mathbf{0\,0}
\,\vline\,\mathbf{0\,1}\,\vline\,\mathbf{1\,0}\,\vline\,\mathbf{1\,1}\,\vline\,1\,0\,\vline\,1\,1\,\vline\,0\,1\,\vline\,1\ldots\\
\hline
\ldots\phantom{\,\vline\,} 1\,1\phantom{\,\vline\,}0\,1\,\vline\,\mathbf{1\,0}\phantom{\,\vline\,}\mathbf{1\,1}\phantom{\,\vline\,}\mathbf{1\,0}\,\vline\,\mathbf{0\,1}\phantom{\,\vline\,}\mathbf{1\,1}\phantom{\,\vline\,}\mathbf{1\,0}\,\vline\,1\,1\phantom{\,\vline\,}0\,0\phantom{\,\vline\,}1\,0\,\vline\,1\ldots \\
\hline
\ldots
\phantom{\,\vline\,}0\,1\phantom{\,\vline\,}0\,0\,\vline\,\mathbf{1\,1}\phantom{\,\vline\,}\mathbf{0\,1}\phantom{\,\vline\,}\mathbf{1\,0}\phantom{\,\vline\,}\mathbf{1\,0}\phantom{\,\vline\,}\mathbf{1\,1}\phantom{\,\vline\,}\mathbf{0\,0}\,\vline\,1\,1\phantom{\,\vline\,}1\,0\phantom{\,\vline\,}0\,1\phantom{\,\vline\,}1\ldots \\
\hline
\ldots
\phantom{\ }1\,1\phantom{\,\vline\,}0\,1\,\vline\ {0\,1}\phantom{i\,}{0\,0}\phantom{i\,}{0\,0}
\phantom{i\,}{0\,1}\phantom{i\,}{0\,1}\phantom{i\,}{0\,0}\,\phantom{!}0\,0
\phantom{\,\vline\,}1\,0\phantom{\,\vline\,}1\,1\phantom{\,\vline\,}0\ldots \\
\hline
\ldots
\end{array}
\]
\caption{The figure shows the first four rows (ordered from the top) of some array $x$ with the markers. 
The boldface symbols form a $3$-rectangle.\label{rect}}
\end{figure}
Once a system of markers is established, it is not difficult to produce a new system with new lower and upper bounds $\pmin_k\le\pmax_k$ of the gap lengths, satisfying 
\begin{enumerate}
	\item[(C)] $\lim_{k\to\infty}\tfrac{\pmin_k}{\pmax_k}=1.$
\end{enumerate}
The new system is obtained from the old one \emph{by subdividing}, i.e., by putting more markers in between the old ones in a way completely determined within each old $k$-rectangle. Here is how it is done: for each $k$ let $m_k$ be such that $m_k(m_k+1)\le \pmin_k$ (the old value). Then every $p\ge\pmin_k$ can be decomposed as a sum $am_k+b(m_k+1)$ with some nonnegative integers $a,b$. We let $a(p)$ and $b(p)$ be the choice of the parameters $a,b$ (for the given $p$) with the maximal possible parameter $a$. Now, in every array $x$ we subdivide each $k$-block (by putting more markers in row $k$) into $a(p)$ sub-blocks of length $m_k$ followed by $b(p)$ sub-blocks of length $m_k+1$, where $p$ is the length of $B$. When this is done, we need to perform the upward adjustment of the newly put markers. The maximal and minimal gaps after the adjustment
lie between $p'^{\mathsf{min}}_k=m_k-(m_1+\dots+m_{k-1}+k-1)$ and $p'^{\mathsf{max}}_k= m_k+1+(m_1+\dots+m_{k-1}+k-1)$. If the numbers $\pmin_k$ (and thus $m_k$) grow fast enough, the condition~(C) will be satisfied for the new system of markers. A system of markers satisfying~(C) will be called \emph{balanced}.
\smallskip

Given a system of markers, the family of all $k$-rectangles occurring in $X$ will be denoted by $\mathcal R_k(X)$. Notice that for $k=1$ $k$-rectangles are the same as $k$-blocks, while for $k>1$, every $k$-rectangle $R$ consists of a concatenation of some finite number $q$ (depending on $R$) of $(k\!-\!1)$-rectangles (this concatenation occupies the top $k-1$ rows) with a $k$-block $B$ appended in row number $k$ at the bottom. We will indicate this by writing
$$
R = \left[\begin{matrix} R^{(1)}R^{(2)}\dots R^{(q)}\\ B\end{matrix}\right].
$$

We will be using the following lemma, which is proved in the last paragraph of \cite{Se12} (although it is not isolated as a separate lemma). The additional property stated below can be easily derived from the construction in \cite{Se12}.
\begin{lem}\label{firstmarker}
Suppose \xt\ is given with a system of markers $(X_0,T_0)$. If $\pmin_{k+1}$ is at least three times larger than $\pmax_k$ for every $k$, then there exists an isomorphic symbolic extension $(Y_0,S_0)$ of $(X_0,T_0)$, over the alphabet $\{0,1\}$. The corresponding factor map has the additional property, that the coding range between $Y_0$ and the $k$-th row of $X_0$ (containing the $k$-markers) does not extend beyond three consecutive $k$-markers. That is, for every $y_0\in Y_0$, the position of every $k$-marker in the image $x_0$ of $y_0$ can be determined by viewing the block of $y_0$ extending at most between the  preceding and following $k$-markers in $x_0$.
\end{lem}

\subsection{Small boundary property} A subset $A\subset X$ of a \tl\ \ds\ \xt\ is a \emph{null set} if it has measure zero for all $\mu\in\mtx$. The system \xt\ is said to have the \emph{small boundary property} if it admits a base of the topology consisting of sets whose boundaries are null sets. Equivalently, the space admits a refining \sq\ (see next line) of finite partitions into measurable sets with null boundaries. A \sq\ of partitions $\{\P_k:k\ge 1\}$ is \emph{refining} if $\P_{k+1}\succcurlyeq\P_k$ for each $k$ and $\diam(\P_k)\to 0$ in $k$, where $\diam(\P)$ denotes the maximal diameter of an atom of a partition $\P$. Obviously, any zero-dimensional system has the small boundary property. Small boundary property can be interpreted as the system being ``equivalent'' to a zero-dimensional one in the following sense:

\begin{fact}(see e.g. \cite{BD05})\label{Gu}
A \tl\ \ds\ \xt\ which has the small boundary property admits an isomorphic zero-dimensional extension \xtp. 
\end{fact}

We recall that a \tl\ factor map $\pi:(X',T')\to(X,T)$ is said to be \emph{isomorphic} if the map $\pi$ is a bijection after discarding some null sets from both $X$ and $X'$. Equivalently, the adjacent map $\pi:\mtxp\to\mtx$ is a bijection (and then an affine homeomorphism) and, for every $\mu'\in\mtxp$ and $\mu = \pi(\mu')$, the standard measure-theoretic systems $(X',\Sigma_{\mu'},\mu',T')$ and $(X,\Sigma_\mu,\mu,T)$ are measure-theoretically isomorphic via the map $\pi$ ($\Sigma_{\mu'}$ and $\Sigma_\mu$ denote the completed Borel sigma-algebras in respective spaces). As we have already mentioned, it is unknown whether the implication in Fact \ref{Gu} can be reversed.

The theorem below follows from works of  E. Lindenstrauss. It says that the class of systems with small boundary property is quite large. Clearly the most interesting for us class of systems admitting an isomorphic zero-dimensional extension is even larger (at least not smaller).
\begin{thm}
If \xt\ has finite \tl\ entropy and has a \tl\ factor which is minimal and aperiodic then \xt\ has the small boundary property.
\end{thm}
\begin{rem}It is unknown whether the existence of an aperiodic minimal factor can be relaxed by only assuming aperiodicity of \xt.
\end{rem}

\begin{rem}Many systems with periodic points also have the small boundary property. For instance, it is so when the system is finite-dimensional while the subset of periodic points has dimension zero (see \cite{Ku95}).
The latter condition holds for example if there are only countably many periodic points, a case which we will exploit most.
\end{rem}

\subsection{Entropy structure and superenvelopes}\label{ss} In a general \tl\ \ds\ the entropy structure is a complicated notion. We will only use the simplified version valid in zero-dimensional systems. Let \xt\ be a zero-dimensional \ds\ represented as the inverse limit of subshifts $(X_k,T_k)$. The \emph{entropy structure} is the \sq\ of functions $\{h_k\}_{k\ge 1}$ defined on $\mtx$ by 
$$
h_k(\mu) = h_{\mu_k}(T_k),
$$ 
where $\mu_k = \pi_k(\mu)$ (here $\pi_k$ is the factor map from $X$ onto $X_k$). The entropy structure has the following properties:
\begin{itemize}
	\item Each $h_k$ is an affine upper semicontinuous function and so are the differences $h_{k+1}-h_k$.
	\item The functions $h_k$ converge (pointwise) nondecreasingly to the entropy function $h$.
\end{itemize}
One of the central notions in the theory of symbolic extensions is that of a \emph{superenvelope of the entropy structure} (or just \emph{superenvelope} for short). This term applies to any function $E$ on $\mtx$ such that $E-h_k$ is nonnegative and upper semicontinuous for all $k\ge 1$. We also admit the constant infinity function as a superenvelope. Every finite superenvelope (if it exists, which is not guaranteed) is upper semicontinuous, hence bounded from above.
Also $E-h$ is upper semicontinuous. 
The pointwise infimum of all superenvelopes is again a superenvelope and we call it the \emph{minimal superenvelope} $\eh$. Clearly, $\eh\ge h$. It is known (see \cite{Do05}) that the equality holds if and only if the entropy structure converges to $h$ uniformly, which is equivalent to the system \xt\ being asymptotically $h$-expansive (in the sense of Misiurewicz, see \cite{Mi76}). Among finite superenvelopes the most important are \emph{affine superenvelopes}, i.e., superenvelopes which are affine functions on $\mtx$. It turns out that $\eh$ equals the pointwise infimum of all affine superenvelopes, hence it is concave, and only sometimes affine. In particular, it is affine in asymptotically $h$-expansive system (because so is $h$).

\subsection{Relations between symbolic extensions and superenvelopes}
The key theorem in the theory of symbolic extensions is the \emph{Symbolic Extension Entropy Theorem}
\cite{BD05}:
\begin{thm}\label{env}
A function $E$ on $\mtx$ equals $h^\pi$ for some symbolic extension if and only if it is an affine superenvelope of the entropy structure of \xt. 
\end{thm}

Here are some immediate consequences of the theorem combined with some facts stated earlier.
\begin{itemize}
	\item $\hsex=\eh$	
	\item $\eh=h^\pi$ in some symbolic extension if and only if $\eh$ is affine.
	\item $\hsex = h$ if and only if \xt\ is asymptotically $h$-expansive. 
	This happens if and only if there exists a \emph{principal} symbolic extension 
	$\pi:(Y,S)\to(X,T)$, i.e., such that $h_\mu(T) = h_\nu(S)$ whenever $\mu = \pi(\nu)$.
	\item \xt\ admits no symbolic extensions if and only if the constant infinity function is the only 
	superenvelope.
\end{itemize}

The above theorem has been refined by J. Serafin in \cite{Se12}: the symbolic extension realizing a finite affine superenvelope $E$ (as $h^\pi$) can be \emph{faithful}, i.e., such that the map $\pi:\msy\to\mtx$ is injective (hence a homeomorphism). In such case $E(\mu)$ equals $h_\nu(S)$ where $\nu$ is the unique preimage of $\mu$.

\subsection{$\mathsf d$-bar distance}
In Theorem \ref{main} of the next section we will use the $\mathsf d$-bar distance introduced by Ornstein \cite{Or74}. By a joining of two measure-preserving systems, we mean a probability measure on the product space  invariant with respect to the product transformation, whose  coordinate projections equal to the original measures. Let $(Y,S)$ be a subshift endowed with two invariant measures $\mu$ and $\nu$. We denote by $J(\mu,\nu)$ the set of all joinings of $\mu$ and $\nu$.  Then the $\mathsf d$-bar distance $\overline{\mathsf d}(\mu,\nu)$ is defined as follows:
$$
\overline{\mathsf d} (\mu,\nu)=\min_{\lambda\in J(\mu,\nu)} \int_{Y\times Y}\mathds{1}_{x_0\neq y_0}(x,y)\,d\lambda(x,y).
$$ 
We recall that the $\mathsf d$-bar distance is stronger than the weak-star topology in the sense that a $\overline{\mathsf d}$-convergent sequence is also weakly-star converging (to the same limit).

By a standard argument, the $\overline{\mathsf d}$-distance has the following convexity property:
for any $\mu,\nu\in\msy$,
\begin{equation}\label{dbar}
\overline{\mathsf d}(\mu,\nu) \le \int\int \overline{\mathsf d}(\xi,\eta)\,dM_\mu(\xi)\,dM_\nu(\eta),
\end{equation}
where $\mu=\int \xi\,dM_\mu(\xi)$ and $\nu=\int \eta\,dM_\nu(\eta)$ are the ergodic 
decompositions of $\mu$ and $\nu$, respectively.

\section{Symbolic extensions with an embedding of aperiodic systems}
We will prove the following refinement of the Symbolic Extension Entropy Theorem:
\begin{thm}\label{main}
Let \xt\ be an aperiodic zero-dimensional \ds\ and let $E$ be an affine superenvelope 
of the entropy structure of $X$. Then there exists a symbolic extension $\pi:(Y,S)\to(X,T)$ such that
\begin{itemize}
	\item $h^\pi = E$.
%	\item The cardinality of the alphabet used in $Y$ is the smallest integer $l$ with 
%	$\log l>\sup\{E(\mu):\mu\in\mtx\}$.
	\item $\diam(\pi^{-1}(\mu))\le E(\mu)-h(\mu)$ for the $\mathsf d$-bar distance on $\msy$.
	\item There exists an equivariant measurable map $\psi:X\to Y$ such that $\pi\circ\psi = \id_X$ (i.e., $\psi$ is a selector from the preimages of $\pi$; such map is necessarily injective), that is to say, $\pi$ is a symbolic extension with an embedding.
\end{itemize}
\end{thm}
The proof is provided at the end of this section. For now, let us list some consequences of the theorem. 

\begin{enumerate}
	\item The last condition implies that every measure $\mu\in\mtx$ has, in its fiber $\pi^{-1}(\mu)$, at 
	least one element, namely $\nu = \psi(\mu)$, such that $(X,\Sigma_\mu,\mu,T)$ and 
	$(Y,\Sigma_\nu,\nu,S)$ are measure-theoretically isomorphic.
	\item If $E(\mu) = h(\mu)$ then $\pi^{-1}(\mu) = \{\nu\}$, where $\nu$ is as above. In other words, $\pi$ 	is faithful and isomorphic on such measures.\footnote{In many situations, the equality $E(\mu) = h(\mu)$ holds on a large set of \im s. For instance, it is known that $\eh = h$ on a residual subset of $\mtx$ and whenever $\eh$ is affine (which is always the case for example whenever the set of ergodic measures is closed), then it is the most natural choice for $E$. More details can be found in \cite{BD05} or \cite{Do05}.}
	\item If \xt\ is asymptotically $h$-expansive then, taking $E=h$, we obtain a symbolic extension which is 
	isomorphic. Because an isomorphic extension is obviously principal (preserves the entropy of each \im),
	we obtain that the following conditions are equivalent in the class of aperiodic zero-dimensional 	
	systems: 
	\begin{itemize}
	\item \xt\ is asymptotically $h$-expansive,
	\item	\xt\ admits an isomorphic symbolic extension. 
	\end{itemize}
	This recovers a result from \cite{Bu16} which sheds a new light on how close asymptotic $h$-expansive 
	systems are to symbolic systems.
	\item If \xt\ is not asymptotically $h$-expansive then the extension described in the theorem cannot be 
	faithful; each measure $\mu\in\mtx$ for which $E(\mu)>h(\mu)$ (and we assume that there are such 
	measures) has in its fiber at least two elements: the measure $\nu$ isomorphic to $\mu$ (hence with 
	entropy $h(\mu)$) and another measure with entropy equal to $E(\mu)$ (since the fiber of $\mu$ is a 
	compact subset of $\msy$, and the entropy function in a symbolic system is upper semicontinuous, the
	supremum $E(\mu)$ of entropy over the fiber is attained). However, the $\mathsf d$-bar distance between these two 
	measures is small if $E(\mu)$ is close to $h(\mu)$.
\end{enumerate}

\begin{rem}\label{smb}The zero-dimensionality assumption in Theorem \ref{main} can be replaced by the property of admitting an isomorphic zero-dimensional extension (in particular, it suffices to assume
the small boundary property). The map $\psi$ will then be defined except on a null set, but then we can use Remark \ref{me} to prolong it. 
\end{rem}

\begin{rem}\label{Hoch}
Our Theorem \ref{main} refines (under the assumptions of the theorem) the result by Hochman \cite{Ho13} on the existence of a finite generator simultaneous for all aperiodic ergodic measures in any Borel \ds\ with finite entropy. However, there is a price for having a uniform generator rather than the simultaneous generator of Hochman. The Theory of Symbolic Extensions (enhanced by this work) implies that a uniform generator must have cardinality at least $2^{\hsex(X,T)}$ \footnote{Throughout this paper we calculate all entropies using the logarithm to base 2 (we will write just ``$\log$'', but in most cases we will not simplify $\log 2$).} and its existence is excluded in systems not admitting symbolic extensions. Hochman's simultaneous generators exist in any aperiodic\footnote{Recently, Hochman extended the result also to systems with periodic points \cite{Ho16}.} \tl\ \ds\ \xt\ and the smallest integer strictly larger than $2^{\htop(X,T)}$ suffices to be the cardinality. 
\end{rem}

The proof of Theorem \ref{main} relies heavily on the original construction of a symbolic extension realizing a given affine superenvelope $E$, as described in \cite{BD05} (see also \cite[Section 9.2]{Do11}). That construction, which we will call \emph{standard}, allows for an almost complete freedom in choosing the family of ``preimage-blocks'' for a $k$-rectangle, as long as the cardinality of this family is kept within a correct range. We will just be more specific about this choice, so that our construction is not even a modification but  (up to a small detail) a \emph{particular case} of the original one. Similar strategy was used by J. Serafin in his construction of a faithful symbolic extension in \cite{Se12}. 
The ``small detail'' which differentiates our construction from the original is the same as used in Serafin's proof: The original construction, which did not attempt to minimize the size of the fiber of a measure, started with replacing the system by its direct product with an odometer. This trick allowed to have a very regular system of markers (all $k$-blocks had the same length), for the price of possibly producing multiple (and usually non-isomorphic) lifts of many \im s, already in this initial step. Like in Serafin's proof, we cannot afford such an extravagance. This is the reason why we are assuming aperiodicity and then we use the natural system of markers built (as a \tl\ factor) into our system. This seemingly affects the entropy estimates (not the construction itself). Fortunately, as we have already remarked, we can always use a balanced system of markers (see condition (C) above) and then the entropy estimates can be conducted as if all $k$-blocks had constant length.

\begin{proof}[Proof of Theorem \ref{main}]
The following two paragraphs summarize the general construction of a standard symbolic extension realizing an affine superenvelope. There is almost nothing new here compared to \cite{BD05} or \cite{Do11}, up to modifications of the system of markers described in \cite{Se12}. 
\smallskip

We start with the system \xt\ given in an array representation. Before we even introduce in \xt\ a system of markers, we need to establish the \sq\ $\{\pmin_k\}$ bounding from below the lengths of the $k$-rectangles. This is done with reference to the relative complexities of the top $k$-row factors and to the given affine superenvelope $E$. The restrictions concern only the speed of growth, more precisely, they establish lower bounds for the values of $\pmin_k$ and of the ratios $\frac{\pmin_{k+1}}{\pmin_k}$.\footnote{Technically, we should bound the ratios $\frac{\pmin_{k+1}}{\pmax_k}$, but since we always have $\pmax_k\le 3\pmin_k$, it suffices to bound the ratios $\frac{\pmin_{k+1}}{\pmin_k}$.} Any sufficiently fast growing \sq\ will serve. Precise inequalities which must be fulfilled are provided e.g. in \cite[Section 9.2]{Do11}. At this point we introduce in \xt\ a balanced system of markers with the bounds $\pmin_k$ and $\pmax_k$. Next, the affine superenvelope $E$ allows us also to determine a finite alphabet $\Lambda$ (we select one element of the alphabet and call it ``zero'') and define an \emph{oracle}, i.e., a \sq\ of functions $\mathcal O_k:\R_k(X) \to \N$, each on the set of all $k$-rectangles appearing in $X$, and satisfying the \emph{oracle inequalities}:
\begin{equation}\label{or1}
\sum_{B\in\mathcal R_1} \mathcal O_1(B) \le (\#\Lambda)^{\pmin_1},
\end{equation}
and, for each $k>1$,
\begin{equation}\label{ork}
\sum_B \mathcal O_k\left(\left[\begin{matrix}
 R^{(1)}R^{(2)}\dots R^{(q)}\\ B
\end{matrix}\right]\right) \le \mathcal O_{k-1}(R^{(1)})\mathcal O_{k-1}(R^{(2)})\cdots\mathcal O_{k-1}(R^{(q)}), 
\end{equation}
where the sum ranges over all $k$-rectangles occurring in $X$ and having, in the top $k-1$ rows, a fixed concatenation $R^{(1)}R^{(2)}\dots R^{(q)}$ of $(k\!-\!1)$-rectangles. The precise values of the oracle depend in a specific way on the superenvelope $E$ (and also on the entropy structure), which in this paper we will skip describing. The interested reader is referred to \cite{BD05} or the book \cite{Do11}. 

Once the oracle is established, one produces a decreasing \sq\ of subshifts $Y_k\subset(\{0,1\}\times\Lambda)^\Z$, together with \tl\ factor maps $\rho_k:Y_k\to X_k$. 
Moreover, the following diagram commutes
$$
\begin{CD}
Y_1 & @<<\id< & 
Y_2 & @<<\id< & 
Y_3 & @<<\id< & \dots\\
@VV\rho_1V & & @VV\rho_2V & & @VV\rho_3V & & \\
X_1 & @<<\pi_1< & X_2 & 
@<<\pi_2< & X_3 & @<<\pi_3< & \dots,
\end{CD}
$$
so the inverse limit $Y$ of the $Y_k$'s (which is simply their intersection, hence a subshift) is a symbolic extension of the inverse limit of the $X_k$'s, i.e., of $X$, via the map $\pi=\lim_k\rho_k$, which on $Y$ happens to be a uniform limit. The alphabet used in each $Y_k$ is $\Lambda^* = \{0,1\}\times\Lambda$
and the subshift is imagined as having two rows: the first row (over $\{0,1\}$) is denoted $Y_0$ and it matches the isomorphic symbolic extension of the system of markers $(X_0,T_0)$ given by Lemma \ref{firstmarker}. The second row in $Y_k$ is more subtle. The rule behind filling the second row (and connecting the oracle with the maps $\rho_k$) is, that with each $k$-rectangle $R$ occurring in $X_k$ (equivalently, in $X$) we associate a ``list'' $\mathcal F_k(R)$ consisting of precisely $\mathcal O_k(R)$ blocks over $\Lambda$ (of the same length as $R$). The families should be disjoint for different $k$-rectangles. Then the preimages by $\rho_k$ of $x_k\in X_k$ have in the second row all possible \sq s $y$ over $\Lambda$ such that ``above'' (i.e., at the same horizontal coordinates as) every $k$-rectangle $R$ in $x_k$ there appears in $y$ a block from $\mathcal F_k(R)$. The map $\rho_k$ is the block-code which replaces each $k$-block $B$ in $(y_0,y)\in Y_k$ (the first row $y_0$ allows to determine the parsing into the $k$-blocks) by the unique $k$-rectangle $R$ such that $B\in\mathcal F_k(R)$. The commutation of the diagram imposes a recursive relation between the lists of order $k$ and $k-1$: if 
$$
R =  \left[\begin{matrix}
 R^{(1)}R^{(2)}\dots R^{(q)}\\ B
\end{matrix}\right]
$$
then the blocks in $\mathcal F_k(R)$ must be selected from the concatenations of the particular form: a member of $\mathcal F_{k-1}(R^{(1)})$ followed by a member of $\mathcal F_{k-1}(R^{(2)})$, and so on, until
a member of $\mathcal F_{k-1}(R^{(q)})$. The oracle inequalities \eqref{or1} and \eqref{ork} make such a selection possible. The dependence of the oracle on $E$ is such that whenever the above described scheme of building $Y$ (and $\pi$) is followed, the extension entropy function $h^\pi$ equals precisely $E$. 
\smallskip

The task in this paper is to provide a specific choice of the families $\mathcal F_k(R)$, which will ensure the last two assertions of Theorem \ref{main}. 
\smallskip

Here is how we proceed. We first modify the alphabet and the oracle (without changing the notation), as follows. We enlarge the alphabet by a few terms to get $\sum_{B\in\mathcal R_1} \mathcal O_1(B) \le (\#\Lambda)^{p_1-2}$ (see \eqref{or1}). Then we replace each value of the oracle, say $\mathcal O_k(R)$, by $(\#\Lambda)^{\lceil\log_{\#\Lambda}(\mathcal O_k(R))\rceil+1}$. Since we have enlarged each value by a factor between $\#\Lambda$ and $(\#\Lambda)^2$ and each $k$-rectangle contains at least two $(k\!-\!1)$-rectangles, \eqref{or1} and \eqref{ork} are still satisfied, so we have created a new oracle, whose values are integer powers of $\#\Lambda$. This modification does not affect any of the entropy computations.

Next we employ a simple combinatorial tool and fact:
\begin{defn}
A \emph{prefix partition} of the family $\Lambda^n$ of all blocks over $\Lambda$ of length $n$ is a partition into cylinders of the form $[C]=\{B\in\Lambda^n: B[1,|C|] = C\}$, where $C\in\Lambda^{|C|}$,
$1\le |C|\le n$. The blocks in one element of the partition have a common prefix $C$, which we will 
call the \emph{fixed positions} and a remaining suffix which ranges over all possible blocks of 
the complementary length, which we will call the \emph{free positions}.
\end{defn}

The proof of the following fact is an elementary exercise and will be omitted.

\begin{lem}\label{pff}
Suppose some positive integers $n_1,n_2,\dots,n_q$ and $n$ satisfy
$$
\sum_{i=1}^q(\#\Lambda)^{n_i}\le(\#\Lambda)^n.
$$
Then there exists a prefix partition $\Lambda^n = [C_1]\cup[C_2]\cup\cdots\cup[C_{q'}]$, with
$q'\ge q$ and satisfying, for all $i\le q$, the equality $|C_i|=n-n_i$ (notice that then $\#[C_i]=\Lambda^{n_i}$).
\end{lem}

We can now specify the families $\F_k(R)$ inductively, as follows. For $k=1$, since the numbers
$\mathcal O(R)$ are integer powers of $\#\Lambda$ summing to at most $(\#\Lambda)^{\pmin_1}$, the families
$\F_1(R)$ can be assigned according to a prefix partition of $\Lambda^{\pmin_1}$: $\F_1(R)=[C(R)]$, satisfying $|C(R)|=\pmin_1 - \log_{\#\Lambda}(\mathcal O_1(R))$. Since the lengths of all blocks in 
$\F_1(R)$ should match that of $R$, and in case $|R|>\pmin_1$ we cannot add any more free positions,
we also fix the $|R|-\pmin_1$ terminal symbols (for example, we put there zeros, regardless of $R$), 
so that the number of the free positions still equals $\log_{\#\Lambda}(\mathcal O_1(R))$.

Suppose for some $k\ge 2$ we have defined the families $\F_{k-1}(R)$ for all $(k\!-\!1)$-rectangles $R$ in such a way that for each $R$ the interval of integers $\{1,2,\dots,|R|\}$ is divided in two subsets, one,
called fixed positions, where all members of the family have the same symbols, and the rest, of cardinaly $\log_{\#\Lambda}(\mathcal O_{k-1}(R))$, called free positions, where all possible configurations occur. We need to specify $\F_k(R)$ for $k$-rectangles $R$. Let
$$
R =  \left[\begin{matrix}
 R^{(1)}R^{(2)}\dots R^{(q)}\\ B
\end{matrix}\right].
$$
According to the rules, we need to select $\F_k(R)$ from the collection of all possible concatenations of blocks belonging to $\F_{k-1}(R^{(1)}),\F_{k-1}(R^{(2)}),\dots,\F_{k-1}(R^{(q)})$ (one block from each family, maintaining the order). Such concatenations have fixed symbols along some set, and range over all possibilities over the rest (the free positions). The number of the free positions equals $\sum_{i=1}^q \log_{\#\Lambda}(\mathcal O_{k-1}(R^{(i)}))$.

The fixed contents determines that if two $k$-rectangles differ already in the top $k-1$ rows then their corresponding families will be disjoint. So, we only need to make sure that the families selected for the  same concatenation $R^{(1)}R^{(2)}\dots R^{(q)}$ (and different last row blocks $B$) are disjoint. 
Denoting 
$$
n(B)= \log_{\#\Lambda}\bigl(\mathcal O_k\left(\left[\begin{matrix}
 R^{(1)}R^{(2)}\dots R^{(q)}\\ B
\end{matrix}\right]\right)\bigr)
$$
the inequality \eqref{ork} takes on the form
$$
\sum_{B} (\#\Lambda)^{n(B)} \le (\#\Lambda)^n,
$$
where $n = \sum_{i=1}^q \log_{\#\Lambda}(\mathcal O_{k-1}(R^{(i)}))$.
Our Lemma \ref{pff} allows to create the families $\F_k(R)$ (with the required cardinalities) according to a prefix partition ``relative on the free positions'' i.e., for every $B$, in addition to the symbols fixed in the previous step for the $(k\!-\!1)$-rectangles, we also fix the symbols along some initial free positions in a way depending on $B$, and allow all possibilities along the remaining free positions (in fact there will be $n(B)$ such free positions). Now the inductive assumption is satisfied for $k+1$. This concludes the construction of the families $\F_k(R)$ and thus of the symbolic extension realizing the prescribed affine superenvelope $E$. 

\smallskip
It remains to check that the extension admits an embedding. To this end we only need to point out a measurable injective and equivariant map $\psi:X\to Y$ which is a selector from preimages of $\pi$. It is obvious that almost every point $x\in X$ (except in a null set) determines two objects: (1) the contents of the first row in all its preimages by $\pi$; this is the symbolic encoding $y_0$ of the system of markers in $x$, and (2) the symbols from $\Lambda$ in the second row along a subset of coordinates (the fixed positions), which is simply the union (over $k$) of the sets of fixed positions corresponding to the $k$-rectangles appearing in $x$. All remaining positions (whose collection may happen to be empty) are admitted all possible configurations of symbols (as $y$ ranges over $\pi^{-1}(x)$). In particular, there appears also the distinguished configuration consisting of all zeros. We assign the corresponding element $y_0\in\pi^{-1}(x)$ to be $\psi(x)$. Measurability of so defined map $\psi:X\to Y$ is standard and the fact that it is equivariant is obvious. 

\smallskip
The proof is almost complete. It remains to estimate, by $E(\mu)-h(\mu)$, the $\mathsf d$-bar diameter of the fiber of an \im\ $\mu\in\mtx$.
The property \eqref{dbar} allows to reduce the problem to the case of $\mu$ ergodic. Then, if $x$ is generic for $\mu$,\footnote{A point is generic for an \im\ $\mu$ if the empirical measures along its orbit converge weakly-star to $\mu$. In symbolic systems, equivalently: the density of occurrences of every block equals its measure. For $\mu$ ergodic the set of generic points has measure $1$.} it is easily seen that the free positions in the preimages of $x$ have density equal to the limit in $k$ of the densities of the free positions in the preimages by $\rho_k$ of $x_k$, which, in turn equal the weighted averages of $\frac1{|R|}\log_{\#\Lambda}(\mathcal O_k(R))$, where $R$ ranges over all $k$-rectangles and the weights are given by the values $\mu$ assigns to the corresponding cylinders. The dependence between $E$ and the oracle (described in the cited earlier works) is such that this limit happens to be exactly $E(\mu)-h(\mu)$. On the other hand, it follows from general facts in \tl\ dynamics that if $x$ is generic for $\mu$ then every ergodic measure in the fiber of $\mu$ has a generic point in the fiber of $x$. By Theorem I.9.10 in \cite{Sh96} two ergodic measures whose some generic points differ along a set of density $\epsilon$ are at most $\epsilon$ apart for the $\mathsf d$-bar distance. This ends the proof.
\qed \end{proof}

\begin{rem}\label{kr}
Krieger (\cite{Kr82}) proved that any aperiodic subshift of \tl\ entropy $h$ is conjugate to a subshift over an alphabet of cardinality $\lfloor2^h\rfloor +1$. So, our symbolic extension \ys\ with an embedding of \xt\ (whose alphabet is a priori $\{0,1\}\times\Lambda$, which is quite large), can be recoded using an alphabet whose cardinality is the smallest integer strictly larger than $2^{\,\sup E}$, where $\sup E$ denotes the largest value of $E$ on $\mtx$.  
\end{rem}

\begin{cor}\label{pg}
If \xt\ is an aperiodic zero-dimensional \ds\ then a uniform generator of \xt\ exists if and only if the minimal superenvelope $\eh$ of the entropy structure is finite (and then it equals $\hsex$, and its maximal value is $\hsex(X,T)$). The optimal cardinality of a uniform generator then equals $\lfloor2^{\hsex(X,T)}\rfloor+1$.
\end{cor}

\section{Extensions preserving selected periodic points}\label{sec4}

In this section we address the symbolic extensions with an embedding (equivalently, uniform generators) for systems with periodic points.
To illustrate the complexity of the problem we begin with a relatively simple yet motivating example.

\begin{exam}\label{example1}
Let $X$ be the array system consisting of all $\{0,1\}$-valued arrays $x$ (with markers) which fit the following description: $x$ has at most two nonzero rows: first of them (if present), number $m$ (which is arbitrary), contains a periodic \sq\ of period $m$ with markers every $m$-th position, while the second nonzero row (if present), number $m+j$ (where $j\ge 2$ is arbitrary), contains a periodic \sq\ of minimal period $mj$ with markers every $mj$-th position aligned with (some) markers in row $m$. The structure of ergodic measures is as follows: there are periodic measures $\mu_{m,i,j,l}$ supported by arrays with nonzero rows $m$ and $m+j$. The minimal period is $mj$. The index $i$ enumerates all possible $\{0,1\}$-valued (with markers) $m$-periodic orbits, and ranges from $1$ to $2^m$. Likewise, the index $l$ enumerates all $\{0,1\}$-valued (with markers) $mj$-periodic orbits, and ranges from $1$ to $2^{mj}$. If we fix $m$ and $i$ and let $j$ grow, the measures $\mu_{m,i,j,l}$ (regardless of $l$) approach the measure $\mu_{m,i}$ supported by arrays with only one nontrivial row, number $m$, containing the $i$-th periodic pattern of period $m$. If we let $m$ grow, the measures $\mu_{m,i}$ approach (regardless of $i$) the pointmass $\mu_0$ of the fixpoint---the zero array. 

%The period structure is as follows: $\mathfrak P_{k}(\mu_{m,i,j,l}) \approx \log 2$ for $k<j$ and $0$ for $k\ge j$. Similarly, $\mathfrak P_{k}(\mu_{m,i}) \approx \log 2$ for $m<k$, and then it drops down to zero. All these functions are zero at $\mu_0$.

%This period tail structure has the same form as the ``pick up sticks game 3'' on page 232 in \cite{Do11} except that single points must be replaced by finite families of measures having common indices $k$ and $j$. %As explained in the book, $u_1$ equals $0$ at all measures $\mu_{k,i,j,l}$ and $\log 2$ at all measures $\mu_{k,i}$ and $\mu_0$, and this function does not repair the structure. The order of accumulation equals 2 and the smallest repair function is $u_2$ which assumes the value $2\log 2$ at $\mu_0$ (see figure on page 234 in \cite{Do11}).

An obvious symbolic extension \ys\ with an embedding is obtained as follows: $Y$ is the subshift over 
four symbols with markers, represented as two-row $\{0,1\}$-\sq s (with markers in each row). The preimage of each array $x$ which has two nontrivial rows is the unique $y$ with the $m$-th row of $x$ copied as the first row and with the $(m\!+\!j)$-th row of $x$ copied as the second row. The arrays $x$ with only one nontrivial row have many preimages: each of them has the first row identical as the $m$-th row of $x$, while the second row contains arbitrary $\{0,1\}$-valued \sq s with either no markers or just one marker
at some place, aligned with a marker in the first row. Finally, the zero array has also many preimages,
each consists of a pair of arbitrary $\{0,1\}$-valued \sq s either with no markers or with one marker
in the first row and no markers in the second, or one marker in both rows at the same place.
The verification of all required properties is straightforward. The function $h^\pi$ equals $0$ on all measures $\mu_{m,i,j,l}$, $\log 2$ on each $\mu_{m,i}$ and $\log 4$ on $\mu_0$. 

But we can create a new, better, symbolic extension with an embedding. And so, if $x$ has two nontrivial rows then any its preimage $y$ will be periodic with the same period $mj$ as $x$. The second row of $y$ will the same as the row $m+j$ of $x$, however, the first row of $y$ will 
be different:
\begin{itemize}
	\item if $j<m$ then the first row of $y$ uses only every $j$-th position, starting at a marker in the 
	second row (other are filled with zeros), where consecutive symbols of the $m$-periodic \sq\ appearing in 
	row $m$ of $x$ are copied (we align the markers in both rows; this row has period $mj$); 
	\item if $j\ge m$ then the first row of $y$ uses only every $m$-th position, where the consecutive $j$ 
	symbols (counting from the marker) of the $m$-periodic \sq\ appearing in row $m$ of $x$ are copied 
	repeatedly, so that the first marker is aligned with a marker in row 2 (this row also has period $mj$ and 
	since $j\ge m$ the entire periodic pattern from row $m$ of $x$ can be	recovered from $y$.
\end{itemize}
To points $x$ with only one nontrivial row we assign one special periodic preimage $y$ whose first row is a copy of the row $m$ of $x$ and the second row is just zeros. This preimage will serve for the embedding. By taking closure of the formerly constructed elements, every such $x$ admits also plenty of other preimages, whose second row is completely arbitrary with no or one marker, while the first row contains the periodic pattern as in row $m$ of $x$ but spread every $m$-th position (creating the period $m^2$). Finally, by taking closure of the formerly constructed elements of $Y$, the zero array in $X$ receives many preimages with an arbitrary second row (with no or one marker), and with the first row which is either just zeros or has one $1$ at some place (it may also have one marker). Also it receives second kind of preimages, with an arbitrary first row (with at most one marker) and trivial second row. Among these preimages there is also the fixpoint of $Y$---the element with two rows filed with zeros and no markers. This preimage serves for the embedding. Each measure $\mu_{m,i,j,l}$ has only one preimage, which is periodic so $h^\pi = 0$ on these measures. The measures $\mu_{m,i}$ have, in spite of periodic lifts, also lifts  supported by \sq s with one arbitrary and one periodic row. For these measures $h^\pi = \log 2$. Finally, $\mu_0$ lifts to a pointmass at the zero array and measures supported by \sq s with one arbitrary row and one trivial row. Here also $h^\pi = \log 2$. We managed to lower the topological entropy of $Y$ to $\log 2$. 
\end{exam}

To handle the general case, we need to develop a rather intricate theory. We begin with simple things.
By $\Per(X,T)$, $\Per_n(X,T)$ and $\Per_{[n_1,n_2]}(X,T)$ we will denote the sets of all periodic points 
in \xt, all periodic points with minimal period $n$, and with minimal period between $n_1$ and $n_2$, respectively. Below we present two crucial notions related to periodic points.

\begin{defn}
The \emph{supremum periodic capacity} and the
\emph{limit periodic capacity} are defined, respectively, as
\begin{gather*}
\psup(X,T) = \sup_{n\ge 1}\tfrac1n\log(\#\Per_n(X,T)),\\
\plim(X,T) = \underset{n\to\infty}{\overline{\lim}}\tfrac1n\log(\#\Per_n(X,T)).
\end{gather*}
\end{defn}
Clearly, $\plim(X,T)\le\psup(X,T)$. If \ys\ is a subshift over an alphabet $\Lambda$ then different $n$-periodic points differ already in the initial block of length $n$, thus $\#\Per_n(X,T)$ estimates from below the number of all blocks of length $n$ occurring in $X$. It is thus elementary to see that 
$$
\psup(Y,S)\le\log(\#\Lambda) \text{\ \ and \ \ }\plim(Y,S)\le \htop(Y,S).
$$
If \ys\ is a symbolic extension of \xt\ with an embedding then \ys\ has periodic capacities at least 
as large as \xt\ has, and thus $\psup(X,T)$ must be finite (in particular, each set $\Per_n(X,T)$ must be 
finite and $\Per(X,T)$ at most countable), the alphabet used in $Y$ must contain at least $2^{\psup(X,T)}$ elements, while its \tl\ entropy must reach at least $\plim(X,T)$. The first fact alone has no further dynamical consequences\footnote{For instance, a system which has $k$ fixpoints cannot be embedded in a subshift over less than $k$ symbols, but otherwise the symbolic extension may have, for example, zero entropy and, except in the fixpoints, use, say, only two symbols.}, but the second one implies that the extension entropy function $h^\pi$ in such extensions may be affected (enlarged) by the structure of periodic points in \xt\ (provided it is rich enough). This is the reason why in Example \ref{example1} we cannot hope to build a symbolic extension with an embedding with topological entropy smaller than $\log 2$ (although $\hsex(X,T)=0$ as in any system with zero \tl\ entropy); the limit periodic capacity in this example is easily seen to equal $\log 2$. Our goal of this section is to describe precisely the dependence of $h^{\pi}$ in symbolic extensions with an embedding on the structure of periodic points, and how it is combined with the usual dependence on the entropy structure.
\smallskip

As already mentioned in the Introduction, we will work in a slightly more general context. We allow the system \xt\ to have arbitrarily many (even continuum) periodic points for every period. However, for each $n$ we select a \emph{finite} and invariant subset $\Per_n^*\subset\Per_n(X,T)$ and we will study symbolic extensions with \emph{partial} embedding, i.e., such that every aperiodic ergodic measure and every periodic measure supported by $\bigcup_n\Per_n^*$ has an isomorphic preimage, equivalently, with an embedding of a Borel set of full measure for any measure as above. The unselected periodic points will typically have no periodic preimages. Clearly, if $\Per_n(X,T)$ is finite for each $n$, choosing $\Per_n^*=\Per_n(X,T)$
we include in our consideration symbolic extensions with (full) embedding as well.

\subsection{The enhanced system}
Let \xt\ be given in an array representation using in row number $k$ the alphabet $\Lambda_k$ ($k\ge 1$). We also attach a row number zero which is formally over the alphabet $\Lambda_0=\{0,0|\}$, but in each $x\in X$ we fill this row exclusively with zeros. For each $n$ we have selected and fixed a finite and invariant subset $\Per_n^*\subset\Per_n(X,T)$. For better understanding of how a symbolic extension with a partial embedding works, it will be helpful to enlarge, in a certain way, the system \xt. We will denote the resulting \emph{enhanced system} by \xtt, and \xt\ is going to be a subsystem (not a factor) of \xtt. Before the definition we need to establish some notation:

\smallskip\noindent
\emph{Notation}: Every periodic point (array) $x\in\Per^*_n$ is an infinite bilateral concatenation of copies of the same \emph{verical strip} extending through $n$ columns and all rows, i.e., a subarray $s\in(\prod_{k\ge 1} \Lambda_k)^n$. Depending on the positioning of the cutting places, $x$ produces $n$ different such strips, on the other hand $n$ points in the same orbit produce the same $n$ strips. Let $\mathcal S_n$ be the collection of all strips obtained in this manner from $\Per^*_n$. Clearly, $\#\mathcal S_n = \#\Per^*_n$. Next, in every strip from $\mathcal S_n$ we put a marker at the rightmost position in the row number zero. By $\mathbf S_n$ we will denote the system consisting of all arrays which are infinite bilateral concatenations of the strips from $\mathcal S_n$. The row number zero of each element $\hat x\in\mathbf S_n$ is $n$-periodic (one marker repeated every $n$-th position). Notice that $\mathbf S_n$ contains $n$ copies of each point $x\in\Per^*_n$, but they differ from $x$ in having markers in row number zero.

\begin{defn}\label{enhanced} We define \xtt\ as the closure of the union $X\cup\bigcup_n\mathbf S_n$ with the action of the usual shift.
\end{defn}

We have the following trivial observations: 
\begin{itemize}
	\item The sets $X$ and $\mathbf S_n$ are pairwise disjoint closed invariant subsets of $\hat X$.
	\item By taking closure of the union $X\cup\bigcup_n\mathbf S_n$ we only add arrays which are limits of 	
	\sq s of arrays belonging to the sets $\mathbf S_n$ with growing parameters $n$. Every such limit 
	array either has no markers in the row number zero, and then it belongs to $X$, or has just one marker and then 
	it is not recurrent. So the ``added set'' $\hat X\setminus(X\cup\bigcup_n \mathbf S_n)$ is contained in 
	the null set of non-recurrent points of \xtt.
	\item If \ys\ is a subshift (has only finitely many nontrivial rows) then so is the enhanced system \yst.
	\item Also notice that $\htop(\mathbf S_n) =\frac1n \log(\#\Per^*_n)$. For this reason, the \tl\ entropy 
	of the enhanced system defined in the following lines is never smaller than the ``partial'' supremum 
	period capacity understood as the supremum of the above entropies over all $n$.
\end{itemize}

The connection between symbolic extensions with (partial) embedding and the enhanced systems is established by the Theorem \ref{seraf} below and Theorem \ref{odwr} in the next subsection. In the proof of the former we will need the following notion and a lemma which refers to it. Note that in the lemma we do not assume \ys\ to be symbolic.

\begin{defn}\label{fin}
A finitary factor map from a \tl\ \ds\ \ys\ to another, \xt, is a continuous equivariant surjection $\pi:Y^\circ \to X^\circ$, where $Y^\circ$ and $X^\circ$ are dense invariant subsets with null complements in $Y$ and $X$, respectively.%\footnote{If necessary, a finitary factor map can be considered defined everywhere, because it can be easily prolonged to a measurable factor map on $X$. The measurable properties do not depend on how it is prolonged. The details are similar as in Remark \ref{me}.}
\end{defn}

\begin{lem}\label{ensex}
Let $\pi:(Y,S)\to(X,T)$ be a factor map between zero-dimensional systems, admitting an equivariant selector $\psi$ from preimages defined at least on the sets $\Per^*_n$. Then there exists a finitary factor map between the enhanced systems, $\hat\pi:(\hat Y,\hat S)\to(\hat X,\hat T)$, such that $\hat\pi|_Y = \pi$, where the enhanced system \yst\ is understood with respect to the sets $\psi(\Per_n^*)$.
\end{lem}

\begin{proof}
For each $n$, $\psi:\Per_n^*\to\psi(\Per_n^*)$ is an equivariant bijection, thus it determines a natural bijection between the set of strips $\mathcal S_n$ (appearing in the elements of $\Per_n^*$) and the analogous set of strips $\mathcal C_n$ appearing in $\psi(\Per_n^*)$ (recall that the strips in both these sets are equipped with a marker at the last position in the row number zero). For $c\in\mathcal C_n$ we let $\hat\pi(c)$ denote the corresponding strip in $\mathcal S_n$.

The factor map $\hat\pi$ will be defined on the union $Y\cup\bigcup_n \mathbf C_n$ as follows:
\begin{itemize}
	\item on $Y$ we let $\hat\pi =\pi$, 
	\item for $\hat y\in \mathbf C_n$, which is a concatenation of some strips $c_j\in \mathcal C_n$ 
	($j\in\Z$), we define $\hat\pi(\hat y)$ as the corresponding concatenation (in the same order and 
	positioning) of the vertical strips $\hat\pi(c_j)$.
\end{itemize}
Obviously, $\hat\pi$ sends $Y\cup\bigcup_n \mathbf C_n$ onto $X\cup\bigcup_n\mathbf S_n$. Continuity of $\hat\pi$ on $Y$ and on each set $\mathbf C_n$ is obvious. If some elements $\hat y_{n_i}\in\mathbf C_{n_i}$ tend to some $y\in Y$ then $\{n_i\}$ must grow to infinity and the strips $c_{n_i}\in\mathcal C_{n_1}$ covering the zero coordinate in $\hat y_{n_i}$ must expand in both directions. This implies that $\pi(y)$ and $\hat\pi(\hat y_{n_i})$ agree on a large rectangle whose both dimensions grow with $i$ to infinity, which means that $\lim\hat\pi(\hat y)=\pi(y)$. We have proved continuity of $\hat\pi$ on its domain.

Because the domain and range of $\hat\pi$ are dense with null complements in $\hat Y$ and $\hat X$, respectively, $\hat\pi$ defines a finitary symbolic extension of \xtt. Clearly, $\hat Y$ contains $Y$, and $\hat\pi|_Y = \pi$, as required.
\qed \end{proof} 

According to the following result, in case \ys\ is a subshift, we can replace the finitary extension by a \tl\ one, without changing the extension entropy function.

\begin{thm}\label{seraf}
Let $\pi:(Y,S)\to(X,T)$ be a symbolic extension with an equivariant embedding of the sets $\Per^*_n$. Then there exists a symbolic extension of the enhanced system, $\breve\pi:(\breve Y,\breve S)\to(\hat X,\hat T)$ such that $h^\pi = h^{\breve\pi}|_{\mtx}$, i.e., $h^\pi$ prolongs to an affine superenvelope of the entropy structure on \xtt.
\end{thm}

\begin{proof}
This is a direct consequence of Lemma \ref{ensex} and a result of Serafin \cite[Theorem 1]{Se09} which says that the family of extension entropy functions in finitary symbolic extensions is the same as the analogous family for continuous symbolic extensions, hence it coincides with the family of all affine superenvelopes of the entropy structure.  
\qed \end{proof}

\subsection{Aperiodic extension of a zero-dimensional system} As a tool leading to reversing Theorem \ref{seraf} we need a specific zero-dimensional principal extension of the enhanced system. Since the construction is general, we formulate it for any zero-dimensional system, but we will apply it later only to \xtt.

\begin{thm}\label{ap}
Let \xt\ be any zero-dimensional system. There exists an aperiodic zero-dimensional extension \xtp\ of \xt, which is isomorphic on aperiodic measures, while each periodic orbit of \xt\ lifts to a collection of odometers.
\end{thm}

\smallskip
Before the proof we establish some notation. Since \xt\ is a subsystem of the \emph{universal zero-dimensional system} \xtu\ defined as the full shift on the Cantor alphabet, which can be modeled as the shift on all arrays using in each row $k$ some alphabet $\Lambda_k$ (we can take any alphabets with $\#\Lambda_k\ge 2$), it suffices to prove the theorem for \xtu. Once we extend \xtu\ to some \xtup\ fulfilling the assertion of the theorem, we can then define \xtp\ as the preimage of \xt\ in \xtup, and obviously this extension will do. The advantage of working with \xtu\ rather than \xt\ is that we have guaranteed the convenient property that the complement of any \emph{almost null} set (see below for definition) is dense. The extension \xtup\ will be built by inserting in rows $k\ge 1$ of the arrays in \xtu\ some markers using what we will call ``almost finitary algorithms'' in the meaning defined below. This is not going to be a system of markers in the sense of Section \ref{zedysy} (although it could be fine-tuned to become one); to achieve aperiodcity a simplified arrangement of markers is sufficient.

\begin{defn}\label{apfin} A measurable subset of a dynamical system will be called \emph{almost null} if it has measure zero for all nonatomic \im s. In other words, the set is a union of a null set and some periodic points. By an \emph{almost finitary factor map} from a \tl\ \ds\ \xt\ to another, \xtp, we will mean a continuous equivariant map $\pi:X_0 \to X'$, where $X_0\subset X$ is a dense invariant subset with 
almost null complement, and $\pi(X_0)$ is dense in $X'$.\footnote{Notice that we do not require $\pi(X_0)$ to have almost null complement in $X'$. In fact, some periodic measures of $X$ may lift to aperiodic meaures supported by this complement.}
\end{defn}

Putting markers in the universal array system \xtu\ using an almost finitary algorithm technically means that given an array $x\in \mathfrak X_0$ (where $\mathfrak X_0$ is an invariant subset of $\mathfrak X$ with almost null complement; its density in the universal system is then automatic), the decision whether a marker should be put at a coordinate $(k,i)$ depends on the contents of $x$ in a finite rectangle around that coordinate. Unlike in the case of continuous algorithms, the size of this rectangle need not be uniformly bounded for coordinates ranging over one row. The dependence on a finite rectangle may fail for $x\notin \mathfrak X_0$, i.e., for arrays which are either periodic or belong to a null set. Once the markers are distributed, we define a new system \xtup\ as the closure of the collection of all arrays with markers. The factor map from \xtup\ to \xtu\ consists simply in erasing the markers and, thanks to the density of $\mathfrak X_0$, this is a surjection. It is obvious that this factor map is an isomorphism on aperiodic measures; the almost finitary algorithm serves as the inverse map. This general scheme does not let us control the lifts of periodic measures, which must be taken care of with help of additional means.

In what follows, in the role of $\mathfrak X_0$ we will be using the set of aperiodic arrays which are recurrent both forward and backward. It is obvious that this set is invariant and has an almost null complement.

\begin{proof}[Proof of Theroem \ref{ap} for \xtu]
To allow markers, we need to enlarge the alphabets to $\Lambda_k^*=\{a,a|:a\in\Lambda_k\}$. 
Recall that if $(k,i)$ and $(k,j)$, $i<j$, are positions of two consecutive markers in row $k$ of some array $x$ with markers then the interval $[i+1,j]$ is called a \emph{gap between markers in row $k$} while the subarray $x([1,k]\times[i+1,j])$ is called the \emph{$k$-rectangle}.\footnote{Since we do not require that the corresponding marker sets are nested, a $k$-rectangle is enclosed by markers only in row $k$, not in every row $1$ through $k$ like on the Figure \ref{rect}.} We will say that the $k$-rectangle \emph{matches a $p$-periodic pattern} whenever $x_{l,n} = x_{l,n+p}$ for all $1\le l\le k$, $i+1\le n\le j-p$, with one exception: the equality $x_{k,j-p} = x_{k,j}$ concerns only the symbols from $\Lambda_k$ ($x_{k,j}$ contains the marker while $x_{k,j-p}$ does not).

To improve readability, we isolate the first major stage of the construction in a separate lemma.

\begin{lem}\label{full}
There exists an almost finitary algorithm defined on the set $\mathfrak X_0$ of recurrent aperiodic arrays, distributing the markers in \xtu, respecting the following rules:

\begin{enumerate}
	\item[(A)] the gaps between markers in row $k$ have lengths at least $k$,
	\item[(B)] if a gap longer than $2k + 1$ occurs then the corresponding $k$-rectangle (we will call it 
	\emph{long}) matches a $p$-periodic pattern with a period $p < k$,
	\item[(C)] for every array $x\in\mathfrak X_0$, there are infinitely many indices $k$ 
	such that the $k$-th row of $x$ contains infinitely many markers going both forward and backward.
\end{enumerate}
\end{lem}

\begin{proof}We begin by recalling the \emph{Krieger's marker lemma} in the version for subshifts with periodic points (see \cite[Lemma 2.2]{Bo83}): In any subshift \ys, for every $K\ge k>1$ there exists a clopen set $F\subset Y$ such that:
\begin{enumerate}
	\item[(a)] $S^{i}(F)$ are pairwse disjoint for $i = 0,1,\dots,k-1$,
	\item[(b)] $X\setminus \bigcup_{i=-k}^{k} S^{i}(F) \subset (\Per_{[1,k-1]}(Y,S))^K$, where 
	the last set consists of points $y\in Y$ such that the block $y[-K,K]$ matches a periodic pattern of  
	some period $p<k$. 
\end{enumerate}

We continue the proof in which we will apply several algorithms. In the first one we distribute Krieger's markers. We proceed inductively, as follows: Step $1$ is idle. Given $k\ge 2$, denote by $(\mathfrak X'_k,\mathfrak T'_k)$ the subshift in the top $k$ rows of \xtu\ with the Krieger's markers already introduced in rows $1$ through $k-1$. Now we apply the Krieger's lemma to $(\mathfrak X'_k,\mathfrak T'_k)$ with the parameters $K=k$ and in this manner we obtain a clopen marker set $F_k\subset\mathfrak X'_k$. We place markers in row $k$ of every array $x\in \mathfrak X'_k$ by the usual rule: if $S^i x\in F_k$ then we place a marker in $x$ at the position $(k,i)$. It is obvious that now the markers satisfy the conditions (A) and (B) (See Figure \ref{Krieger}). Because the sets $F_k$ are clopen, this algorithm is continuous at every point.
\begin{figure}[ht]
\begin{center}
\includegraphics[width=10cm]{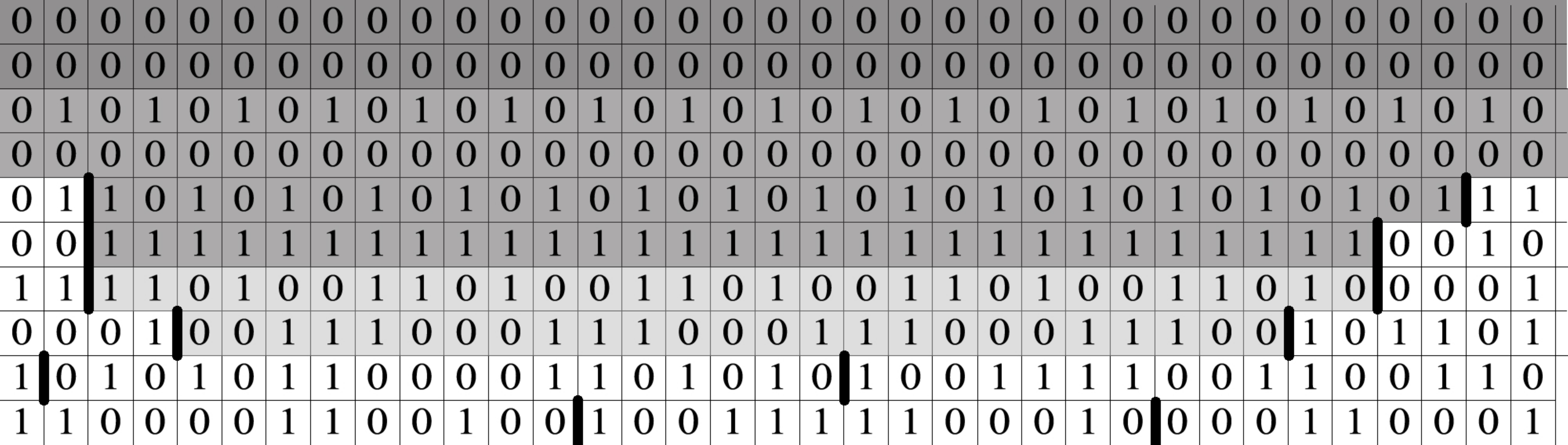}
\caption{\small\label{f0}Placement of Krieger's markers. The gray areas contain periodic patterns with periods 1, 2 and 6. In rows 1 through 4 the markers are missing (infinitely long gaps). The gaps is rows 5,6,7 and 8 are finite but long. \label{Krieger}}
\end{center}
\end{figure}

We need another algorithm in which we place so-called \emph{periodic markers}. They are meant to ensure that arrays belonging to so-called \emph{slow odometers} satisfy (C). An odometer consist of arrays in which all rows are periodic with unbounded minimal periods. An odometer is \emph{slow} if for all except finitely many indices $k$, the periodic pattern in the top $k$ rows has minimal period less than $k$. In arrays belonging to such odometers, although they are aperiodic, Krieger's markers appear only in finitely many rows. This must not be admitted and here is what we do:

Again we proceed by induction, and skipping the first step we pass to $k\ge 2$. As before, we let $(\mathfrak X'_k,\mathfrak T'_k)$ denote the subshift in the top $k$ rows of \xtu\ with all the Krieger's markers already introduced. Each periodic orbit in $\mathfrak X'_k$ with minimal period $p<k$ can be identified with a periodic pattern in the top $k$ rows, understood up to shifting. For every such pattern we select and fix one out of $p$ possible ways of distributing in this pattern some future markers periodically: one for every $p$ positions. Once this is established, in every element $x\in\mathfrak X'_k$ we search for long $k$-rectangles. As we know, every such rectangle matches a $p$-periodic pattern of minimal period $p<k$. Now, within each long $k$-rectangle we place new markers one every $p$ positions exactly as it was decided for the corresponding pattern, but not in row number $k$ only in row number $p$, and skipping all these markers which would fall closer than $p$ positions away from any markers already put in that row.\footnote{This practically means that either there already are some markers (Krieger's or periodic put in a preceding step) appearing one every $p$ positions, and then we do not put any markers in this step, or there are only Krieger's markers outside the long $k$-rectangle and we need to mind only the two closest external ones.} This concludes the $k$-th step of the induction
(see Figure \ref{per-aper}).
\begin{figure}[ht]
\begin{center}
\includegraphics[width=10cm]{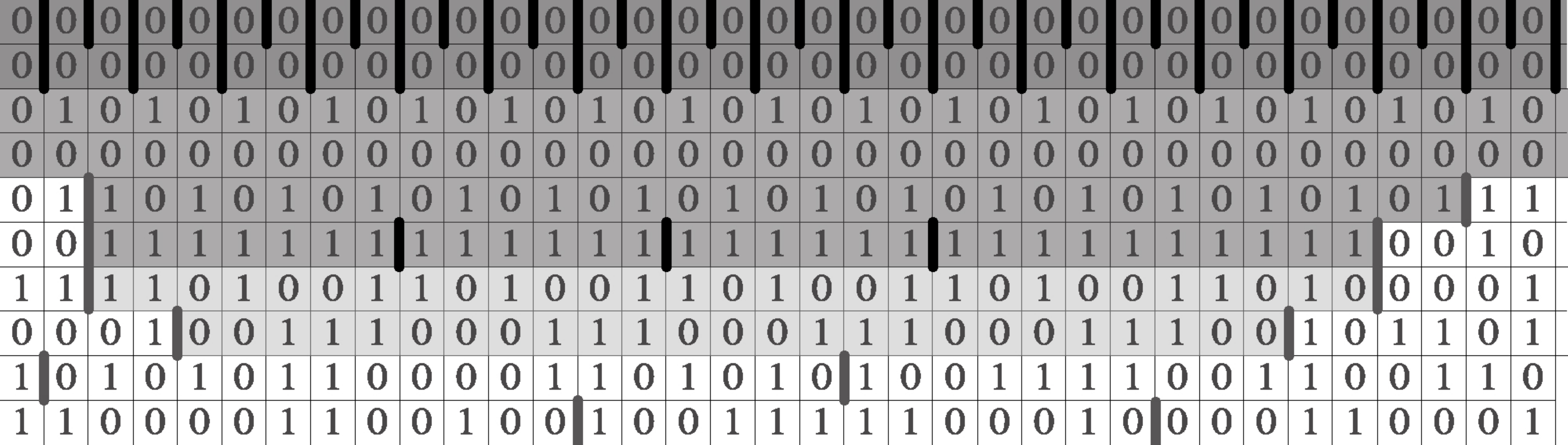}
\caption{\small\label{f1}Periodic markers in rows 1, 2 and 6. Note that periodic markers fall in a row $k>1$ only when the period of the top $k$ rows (within some long rectangle) is strictly less than $k$, 
thus their placement \emph{increases} that period. This is how we get rid of slow odometers (still, each slow odometer is replaced by a conjugate model).\label{per-aper}}
\end{center}
\end{figure}

Once the induction is completed, the markers obviously satisfy the conditions (A) and (B). We will prove also (C). Consider an array $x\in\mathfrak X_0$ with all the markers put in so far. As we will show in the next paragraph, each marker is put in $x$ by a rule continuous on the invariant set $\mathfrak X_0$. Since $x$ is forward and backward recurrent, each row of $x$ contains either infinitely many markers (in both directions) or no markers at all. If (C) fails, markers must be completely missing in all sufficiently far rows of $x$. But then, by (B), for every $k$, the contents of the top $k$ rows is periodic. If the periods were unbounded as $k$ increases, we would have periodic markers in infinitely many rows. So the periods are bounded and thus $x$ is periodic hence cannot belong to $\mathfrak X_0$. This way or another we arrive at a contradiction. 

We shall now analyze the discontinuity points of the second algorithm, i.e., identify the arrays of $x\in\mathfrak X_0$ in which some periodic markers cannot be predicted by viewing bounded rectangles. 
We focus on a coordinate $(k,i)$ at which there is no marker in $x$ (once a marker is put in some step of the induction, which happens with reference to a bounded rectangle, it is never removed afterwards). In order to determine periodic markers in row $k$ near the position $i$ we need to look at larger rectangles in the top $k'$ rows for $k'>k$. Once the coordinate $i$ falls in a gap in row $k'$ shorter than or equal to $2k'+1$ (but never shorter than $k'$), then we are sure that periodic markers will not appear in row $k$ within this gap, because in step $k'$ periodic markers do not apply, while in all further steps the smallest period of any periodic pattern containing this gap in row $k'$ is at least as large as the gap, i.e., not less than $k'$, so the periodic markers would go to a row with number at least $k'$. On the other hand, if for each $k'\ge k$ the coordinate $i$ falls in a long gap then the corresponding $k'$-rectangles contain periodic patterns with nondecreasing periods. Once the period reaches or exceeds $k$, we can stop looking further: the periodic markers in row $k$ (near the coordinate $i$) are already decided. The only case in which we are ``never sure'', is when $i$ falls in long $k'$-rectangles for all $k'>k$ and each time the period of the corresponding pattern is smaller than $k$. Such an array $x$ is obviously periodic on at least one side of the coordinate $i$. So, $x$ is either periodic or belongs to the null set of nonrecurrent points, i.e., does not belong to $X_0$, which ends the proof. 
\qed \end{proof}

\begin{rem}
If we apply this proof to a system \xt\ such that $\Per_n(X,T)$ is finite for every $n$, then all periodic markers can be predicted by looking at finite areas, hence the above algorithm is in fact continuous.
\end{rem}

We continue with the proof of Theorem \ref{ap} by applying two further almost finitary algorithms to put even more markers in the arrays with markers obtained in the preceding lemma. This time we aim to producing markers satisfying
\begin{enumerate}
	\item[(D)] for every array $x\in\mathfrak X_0$, there are infinitely many markers in every row, going 
	both forward and backward,
	\item[(E)] the gaps in row $k$ have lengths ranging between $k$ and $2k-1$.
\end{enumerate}

First we apply so-called \emph{upward stretching}: we copy every marker from row $k$ to all rows with indices $k-1, k-2,...$ until we reach a row $l<k$ in which the marker would fall $l$ or less positions away from some marker (then we do not copy it to rows with indices $\le l$), or till row $1$ (see Figure \ref{upwarda}). 
\begin{figure}[ht]
\begin{center}
\includegraphics[width=10cm]{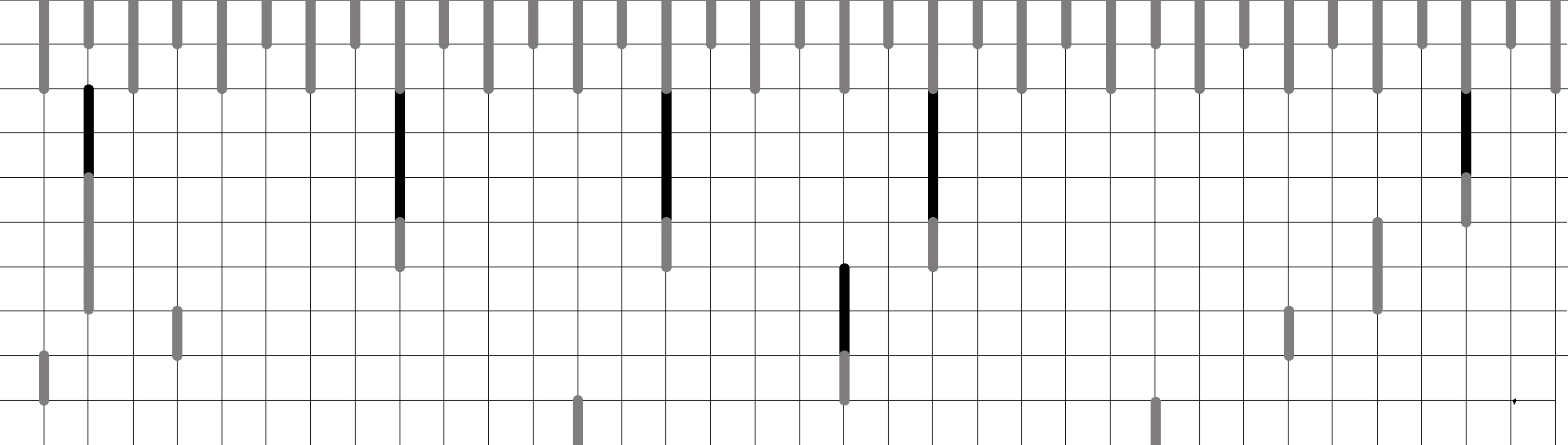}
\caption{\small\label{f2} The upward stretching. New markers shown in black.\label{upwarda}}
\end{center}
\end{figure}
In this manner we maintain the property (A) (which is included in (E)). Clearly, because of (C), after applying this algorithm we have satisfied (D). Let us analyze the discontinuities. Suppose $(k,i)$ is a position in an array $x$ where, by looking at no matter how large finite area, we always admit a marker coming later by upward stretching from some far row. This is possible only if there are no markers in the ``triangular area'' between $i-k'$ and $i+k'$ for all rows $k'>k$; every marker in this area would stop any potential upward stretching arriving to $(k,i)$ (see Figure \ref{upwardb}). 
\begin{figure}[ht]
\begin{center}
\includegraphics[width=10cm]{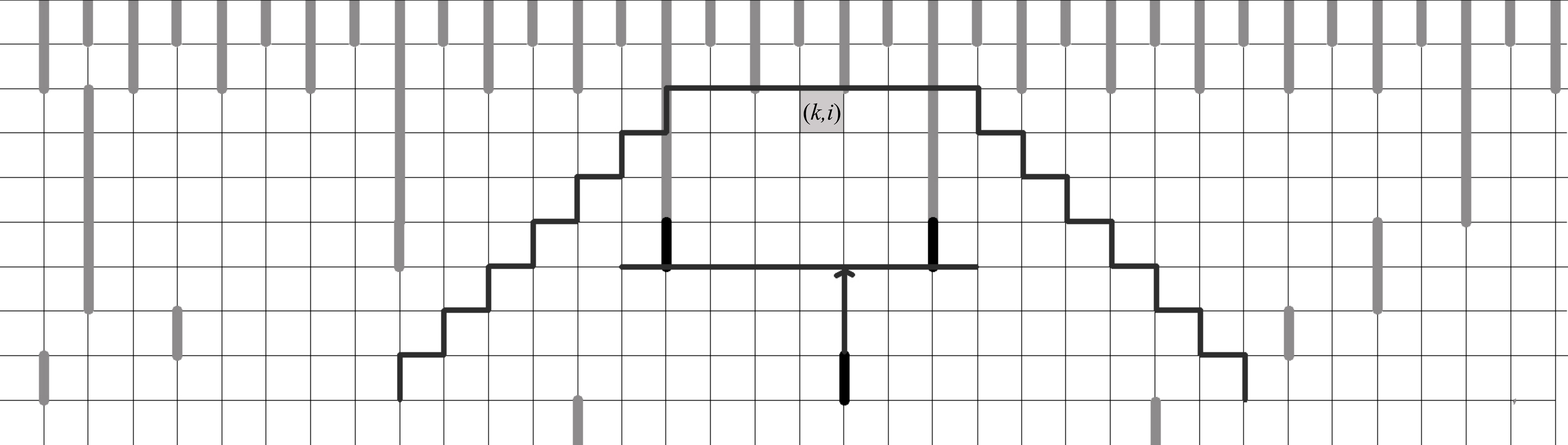}
\caption{\small\label{f3}The ``triangular area'' and markers blocking the upward stretching toward the position $(k,i)$.\label{upwardb}}
\end{center}
\end{figure}
Since the width of the triangular area in row $k'$ equals $2k'+1$, this means that $i$ belongs to a long gap in row $k'$, and thus to a long $k'$-rectangle, and these rectangles expand in both directions as $k'$ grows. Thus, by (B), $x$ is either periodic, or belongs to an odometer. But in every odometer there exists a row $k'$ such that the period $p$ of the pattern in the top $k'$ rows is larger than $k$, resulting in periodic markers in row $p>k$. At least two of these markers would appear in the above mentioned triangular area, which excludes this case. We have shown that the algorithm is discontinuous only at periodic arrays, hence it is almost finitary.
\smallskip

The last algorithm, which we are about to apply, will reduce the gap sizes in every row $k$ to at most $2k-1$ (without decreasing the lower bound $k$) as required in~(E). We call it the \emph{leftward stretching}: for every $k$, every marker in row $k$ appearing at a position $i$ is copied in the same row at positions $i-k, i-2k$, etc., until we arrive within less than $k$ positions away from some marker in this row (or till minus infinity)  (see Figure \ref{left}). 
\begin{figure}[ht]
\begin{center}
\includegraphics[width=10cm]{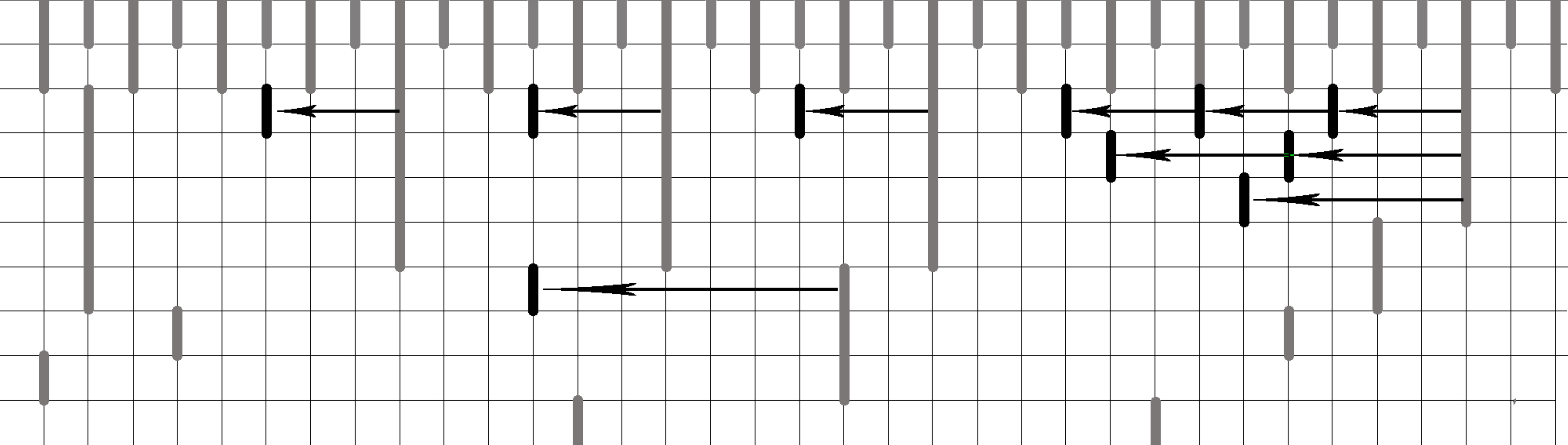}
\caption{\small\label{f4}Leftward stretching. New markers shown in black.\label{left}}
\end{center}
\end{figure}
It is obvious that the algorithm reduces the gap sizes as required. Also, thanks to (D), it is continuous at every point of $\mathfrak X_0$, hence almost finitary. This ends the description of the algorithm of placing markers in the universal system.

\smallskip
We recall the at this moment we define the extension \xtup, as follows: We take all arrays from the (dense) set $\mathfrak X_0$ on which the algorithm is continuous and produces arrays with markers satisfying (D) and (E). It is clear that these properties pass to the closure of the set of so created arrays with markers, and that arrays with so distributed markers are never periodic. So, letting $\mathfrak X'$ be the closure of this collection of arrays with markers, we obtain an aperiodic system. The factor map consisting in erasing all markers sends $\mathfrak X'$ onto $\mathfrak X$ (here we use the density  of $\mathfrak X_0$ in $\mathfrak X$). This map is invertible on $\mathfrak X_0$; the algorithm of placing the markers serves as the inverse map and, by continuity, $\mathfrak X'$ contains no points projecting to $\mathfrak X_0$ other than those obtained by this algorithm. This implies that \xtup\ extends \xtu\ isomorphically for all aperiodic measures. It remains to examine the lifts of periodic arrays (we need them to be elements of odometers, i.e., have all rows periodic). 
\smallskip

Let $x$ be a periodic array of period $p$ in the system \xtu\ and let $\{x_j:j\ge 1\}\subset\mathfrak X_0$ be a \sq\ approaching $x$. We assume that the arrays $x_j$ equipped with all due markers (denoted $x_j'$) converge to some $x'\in\mathfrak X'$. Then $x'$ is a lift of $x$ and all lifts of $x$ are obtained in this manner. Given $k$, we will analyze the markers in the ``test interval'' $[-m,m]$ in row $k$ of $x'_{j_m}$, for some $m$ and $j_m$ so large such that the array $x_{j_m}$ matches $x$ on $[1,2m]\times[-2m,2m]$. In particular, in this rectangle $x_{j_m}$ is periodic with the period $p$. If $k<p$, $x'_{j_m}$ has, in the test interval, markers occurring $p$-periodically. Consider the case $k\ge p$. There are several possibilities for $x'_{j_m}$: 
\begin{enumerate}
	\item The ``triangular area'' $\bigcup_{k'\ge m}(\{k'\}\times[-k',k'])$ contains no Krieger's 
	markers.\footnote{Although the marker symbols do not allow to distinguish between Krieger's, 
	periodic and stretched markers, we can always determine Krieger's markers by removing all markers and 
	repeating the first (continuous) algorithm.} By the rule (B), which is satisfied by Krieger's markers 
	alone, $x_{j_m}$ represents a slow odometer. We will come back to this case later.
	\item Otherwise let $k_0$ be the smallest index $k'\ge m$ such that a Krieger's marker is inserted within 
	$[-k',k']$ in row $k'$ of $x'_{j_m}$. Since the rows $k$ through $2m$ of $x_{j_m}$ are $p$-periodic along 
	$[-2m,2m]$, there are no Krieger's markers there, and thus $k_0>2m$. As before, the rectangle 
	$[1,k_0-1]\times[-k_0+1,k_0-1]$ of $x_{j_m}$ still contains a periodic pattern with a period 
	$p_0\in[p,k_0-1]$. Now we have two subcases:	
	\begin{enumerate}
		\item If $p_0\ge k$ then periodic markers occur in row $p_0$ across $[-k_0-1,k_0-1]$. 
		These markers prevent any markers in further rows from affecting the test interval by upward 
		stretching, i.e., the markers in the test interval (in $x'_{j_m}$) depend only on the (periodic) rows 
		$1$ through $k_0-1$ of $x_{j_m}$, just as if $x_{j_m}$ belonged to a slow odometer.
		\item If $p_0<k$ then the test interval has no Krieger's markers, no periodic markers of its own period 
		and no markers upward stretched from rows up to $k_0-1$. It may (but need not) receive a marker 
		upward stretched from row $k_0$ or further and clearly such a marker is then unique. We will call it 
		the \emph{intrusion}. Then, in spite of the intrusion, the test interval will eventually have only 
		markers generated by leftward stretching, which will occur precisely one every $k$ positions, with one 
		possible larger gap to the right of the intrusion (but still, not larger than $2k-1$; see Figure 
		\ref{intrusion}).
	\end{enumerate}
\end{enumerate}

\begin{figure}[ht]
\begin{center}
\includegraphics[width=10cm]{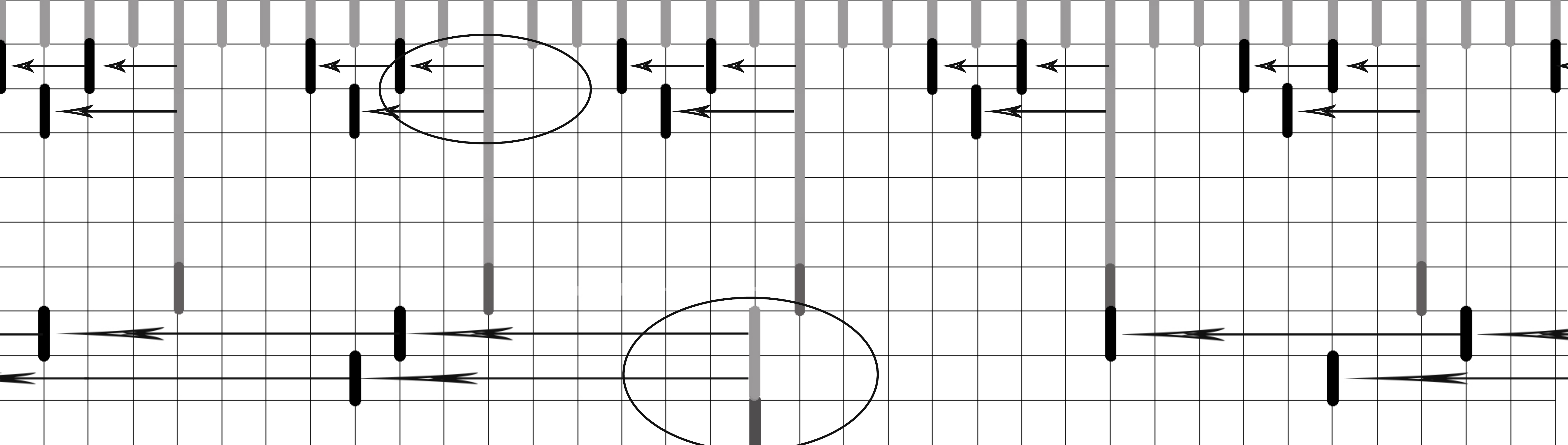}
\caption{\small\label{f5}An array which up to row 9 matches a $7$-periodic pattern. The marker in row 10 creates, by upward stretching, intrusions in rows 9 and 8. In row 8 (accidentally) all gaps have the same length 8, in row 9 there appears an exceptional larger gap. Periodic markers in row 7 generate intrusions in rows $6, 5,\dots, 2$, but only in rows 2 and 3 there is enough room for leftward stretching.\label{intrusion}}
\end{center}
\end{figure}

If for all $m$ we have the case (2b), then $x'$ has in row $k$ markers appearing $k$-periodically
with one possible larger gap, hence either the $k$-th row of $x'$ is periodic with the period $\mathsf{LCM}(p,k)$ or $x'$ is not recurrent.

It remains to see what happens in row $k$ of $x'_j$ when $x_j$ belongs to a slow odometer. Then in row $k$ we have either periodic markers at every $k$-th position and no markers otherwise, or stretched up periodic, occurring one every $k'$ positions for some $k'\ge k$, and leftward stretched subdividing the gaps between the periodic markers into gaps of lengths $k$ and one perhaps longer gap, up to $2k-1$ (only if $k'\ge 2k$; this situation is seen on Figure \ref{intrusion} for example in row $2$). If the parameter $k'$ grows with increasing $j$, then eventually there will be at most one longer gap in the $k$-th row of $x'$ and we are back in the case described earlier. Finally, if $k'$ is bounded (then we can assume it is constant) as $j$ increases, then the row $k$ in $x'$ has markers occurring with period $k$ (if $k'$ is a multiple of $k$) or $k'$, and the $k$-th row of $x'$ is perioidc with the period $\mathsf{LCM}(p,k)$ or $\mathsf{LCM}(p,k')$. Summarizing, $x'$ is either not recurrent, or all its rows are periodic (with unbounded periods), hence it belongs to an odometer. This concludes the proof.
\qed \end{proof}

\subsection{Building a symbolic extension with partial embedding}
We can now prove a theorem converse to Theorem \ref{seraf} (even a bit stronger than verbatim the converse, because in that theorem we did not assume an embedding of aperiodic measures). Recall that we have fixed the finite sets $\Per^*_n$ and meaning of the enhanced system depends upon this choice. Also recall that by a partial embedding we mean an equivariant measurable selector from preimages defined except on an almost null set being a union of a null set and the periodic points \emph{not} belonging to~$\bigcup_n\Per^*_n$.

\begin{thm}\label{odwr}
Suppose $\breve\pi:(\breve Y,\breve S)\to (\hat X,\hat T)$ is a symbolic extension of the enhanced system. Then there exists a symbolic extension $\pi:(Y,S)\to(X,T)$ with partial embedding and such that 
$h^\pi = h^{\breve\pi}|_{\mtx}$.
\end{thm}

\begin{proof} 
We let $\pi':\hat X'\to\hat X$ be the natural factor map (deleting the markers) from the aperiodic extension constructed in Theorem \ref{ap} for \xtt. We will use the fact that $\pi'$ is a ``one-block code'', i.e., that each symbol in the image array depends only on the corresponding one symbol in the source. At this point it will be convenient to completely forget that the extension \xttp\ is equipped with some markers. We will soon need to introduce in \xttp\ an entirely new system of markers, hence remembering the old ones would only obfuscate the picture. We only need to remember that \xttp\ is an aperiodic, isomorphic on aperiodic measures and principal (periodic measures lift to odometers) zero-dimensional extension of \xtt\ represented as an array system and that the factor map $\pi'$ is a one-block code. Recall that the arrays in $\hat X$ have a special row number zero containing, in arrays belonging to $\mathbf S_n$, some $n$-periodically repeated single markers, otherwise it is empty or contains just one marker. We can attach this row number zero to the elements of \xttp\ (in $\hat x'\in\hat X'$ we copy the row number zero from $\pi'(\hat x')$). The markers in the row number zero will be called the \emph{dominant markers} (they have nothing to do with the markers introduced while building \xttp). We denote $\mathbf S_n' ={\pi'}^{-1}(\mathbf S_n)$. Clearly an array $\hat x\in\hat X'$ belongs to $\mathbf S_n'$ if and only if it has markers in the row number zero at every $n$-th position.

Now, we are given a symbolic extension $(\breve Y,\breve S)$ of \xtt, which realizes some affine superenvelope $\hat E$ of the entropy structure of \xtt\ as the extension entropy function. Since $(\hat X',\hat T')$ is a principal extension of \xtt, $(\breve Y,\breve S)$ can be principally extended to a symbolic extension of $(\hat X',\hat T')$. This fact is proved for example in \cite[Theorem 7.5]{BD05}. Thus, by change of notation, we can replace $(\breve Y,\breve S)$ by this new symbolic extension of $(\hat X',\hat T')$, because it is also a symbolic extension of \xtt\ and yields the same extension entropy function $\hat E$. Since $(\hat X',\hat T')$ is aperiodic, by Theorem \ref{main} there exists another symbolic extension of $(\hat X',\hat T')$, this time \emph{with an embedding}, which has the same extension entropy function (the lift of $\hat E$ to $\M_{\hat T'}(\hat X')$). So, we can assume that $(\breve Y,\breve S)$ is such an extension. We have the following factor maps $\breve\pi':\breve Y\to\hat X'$, \ $\pi':\hat X'\to \hat X$ \ and \ $\pi'\circ\breve\pi' = \breve\pi:\breve Y\to\hat X$.

The factor map $\breve\pi$ leading from $(\breve Y,\breve S)$ to \xtt\ and restricted to the preimage of \xt\ provides a symbolic extension with an embedding of aperiodic measures and with the extension entropy function equal to the restriction of $\hat E$ to $\mtx$. The only remaining problem is to include in this extension, for each $n$, $n$-periodic points which would map injectively and onto $\Per^*_n$, without increasing the entropy function of the extension (measurability is trivial, as we are including a countable set). 

The factor map $\breve\pi'$ from $(\breve Y,\breve S)$ of \xttp\ is obtained via Theorem \ref{main}, hence it is of special form, which we call \emph{standard}, that is, it refers to some $k$-rectangles, equivalently, to a system of markers (in the meaning of Section \ref{zedysy}), and then uses some disjoint ``lists of preimage blocks'' of these $k$-rectangles. As we have mentioned earlier, the system of markers used for this purpose must satisfy two conditions:
\begin{enumerate}
	\item[(a)] the \sq\ $\{\pmin_k\}$ must grow sufficiently fast, where the speed is determined by the 
	system, the choice of the partitions defining the array representation, and by the affine superenvelope 
	$\hat E$,
	\item[(b)] be balanced (i.e., $\frac{\pmin_k}{\pmax_k}$ must tend to $1$).
\end{enumerate}
We will need the following  

\begin{lem}\label{uuu} In the aperiodic system \xttp\ equipped with the row number zero containing the dominant markers there exists a system of markers satisfying the conditions (a), (b) and (c) formulated below, for some tending to infinity \sq\ of integers $k(n)$:
\begin{enumerate}
	\item[(c)] if $\hat x\in\mathbf S'_n$ then each dominant marker copied to rows $1$ through $k(n)$ is used 
	by the system of markers
\end{enumerate}
(in other words, the rectangle in the top $k(n)$ rows between two dominant markers is a concatenation of complete $k(n)$-rectangles).
\end{lem}
\begin{proof}
We only outline the proof, which consists in two steps, both being continuous algorithms of placing and adjusting the markers. 

Step 1: we choose a \sq\ $\{m_k\}$ such that the numbers $m_k-(m_1+\dots+m_{k-1}+k-1)$ satisfy (a) and 
we let $n_k=2m_k(m_k+1)$. Then we let $k(n)$ be the unique integer $k$ such that $n_k\le n<n_{k+1}$. Now, in arrays belonging to $\mathbf S'_n$, we copy all the dominant markers in rows with numbers $k\in\{1,\dots,k(n)\}$. Next, we apply the Krieger's markers skipping those which fall (in their rows $k$) closer than $n_k$ positions away from the copied dominant markers. Finally, we apply the upward adjustment, which clearly does not move the dominant markers, because they are already upward adjusted. In this manner we obtain a preliminary system of markers with $\pmin_k\ge \frac{n_k}2\ge m_k(m_k+1)$ (at this point $\pmax_k$ is just finite).

Step 2: By subdividing (see the description few lines below Figure \ref{rect}) we obtain a new, denser (thus still satisfying (c)) and balanced (as required in (b)) system of markers. All rectangle lengths are now bounded below by $m_k-(m_1+\dots+m_{k-1}+k-1)$, hence the new system of markers satisfies also (a).
\qed \end{proof}

From now on, we can assume that $\breve\pi':\breve Y\to\hat X'$ is a standard symbolic extension of \xttp\ built using a system of markers satisfying (a), (b) and (c). Recall also that the standard symbolic extension consists of elements $\breve y$ with two rows, the first one responsible for memorizing the positions of all markers in $\hat x'=\breve\pi'(\breve y)$, and the second one, where the ``preimage blocks'' occur. The advantage of a standard symbolic extension over an arbitrary symbolic extension is that a $k$-block in $\breve y$ (lying between two markers of order $k$) determines the entire $k$-rectangle lying underneath in $\hat x'$, without missing the margins, as it may happen in a general symbolic extension (see Remark \ref{piecszesc} and Figure \ref{gluing}). Unfortunately, in order to locate the markers of order $k$ we may need to see not only the $k$-block between them in the first row of $\breve y$, but also some context. We will need to take care of this small inconvenience.   

Define $\breve{\mathbf S}_n\subset \breve Y={\breve\pi}'\vphantom{ }^{-1}(\mathbf S'_n) = {\breve\pi}^{-1}(\mathbf S_n)$. These subsystems are pairwise disjoint (because so are the $\mathbf S_n$'s). Now we apply a very simple trick: we equip each $\breve y\in\breve Y$ with an extra row (number zero) which is a copy of the row number zero of $\breve\pi'(\hat y)$ (the same as the row number zero in $\breve\pi(\breve y)$, and which is essentially nontrivial only if $\breve y\in\breve{\mathbf S}_n$). This trick is clearly a conjugacy, but it allows to locate the dominant markers in $\breve Y$ without referring to any context. Let \emph{$n$-words} be the blocks of length $n$ appearing in $\breve{\mathbf S}_n$ and ending with a marker in the row number zero. Now, whenever in an element $\breve y\in\breve Y$ we see a block $B$ of length $n$, ending with a marker in the row number zero \emph{and directly preceded} on its left by another marker in this row, then: 
\begin{itemize}
	\item We know that $\breve y\in\breve{\mathbf S}_n$ and that $B$ is an $n$-word.
	\item Using the first row of $B$ (containing the encoded markers), we can determine \emph{without any 
	further context} all markers of orders $1$ through $k(n)$ that fall within $B$---this is a consequence of 
	Lemma \ref{firstmarker}.
	\item We can thus determine (also without any further context) the complete contents of the top $k(n)+1$ 
	rows of $\breve\pi'(\breve y)\in\hat X'$ (counting also the row number zero), along the $n$ horizontal positions 
	occupied in $\breve y$ by $B$---this is due to the construction of a standard symbolic extension 
	$\breve\pi'$ and the condition (c) of Lemma \ref{uuu}, according to which, this part of 
	$\breve\pi'(\breve y)\in\hat X'$ is a concatenation of some complete $k(n)$-rectangles. Next, applying 
	the one-block code $\pi'$ we determine the contents (call it $R(B)$) of the same area in 
	$\breve\pi(\breve y)\in\hat X$.
	\item Furthermore, since $\breve y\in\breve{\mathbf S}_n$ we know that $\breve\pi(\breve y)\in\mathbf 
	S_n$. Since $B$ is enclosed between a pair of dominant markers, the reconstructed rectangle $R(B)$ is 
	part of some vertical strip $s\in\mathcal S_n$.
\end{itemize}

We continue as follows: We fix some $n$ and for each strip $s\in\mathcal S_n$ let $\mathcal W_n(s)$ be the family of all $n$-words which occur ``above'' $s$ in preimages by $\breve\pi$. These sets are obviously nonempty, and need not be disjoint for different strips, but only when the strips agree in the top $k(n)$ rows. We will now find an injective selector from the families $\mathcal W_n(s)$.\footnote{This is the place in the proof where it becomes essential that $(\breve Y,\breve S)$ extends not just \xt\ but also the enhanced system \xtt. By referring to the \tl\ entropies of the subsystems $\mathbf S\subset \mathbf S_n$ we implicitly involve some ``local'' limit periodic capacities of \xt.} Observe that these families satisfy the Hall's marriage condition: for any nonempty subset $\mathcal S\subset\mathcal S_n$, let $\mathbf S$ be the collection of all free concatenations of the strips from $\mathcal S$. Clearly, $\mathbf S$ being a subsystem of $\mathbf S_n$ has a preimage in $\breve Y_n$ of entropy not smaller than that of $\mathbf S$. 
Because $\mathbf S$ uses all possible concatenations of its ``alphabet'' $\mathcal S$, while the above preimage uses at most all concatenations of its ``alphabet'' $\bigcup_{s\in\mathcal S}\mathcal W_n(s)$, the latter ``alphabet'' must be at least as numerous as the former. The Hall's Marriage Theorem now provides an injection assigning to every strip $s\in\mathcal S_n$ an $n$-word $W_s\in\mathcal W_n(s)$. At this point, for each $n$ and each $s\in\mathcal S_n$ we add to $\breve Y_n$ the $n$-periodic orbit of the \sq\ created by repetitions of the $n$-word $W_s$, and we define the factor map on this orbit by sending it onto the $n$-periodic orbit of the point $x\in\hat X$ consisting of repetitions of the strip $s$. The map should preserve the positioning of the dominant markers. By taking closure, the resulting enlarged symbolic space, denoted by $\breve Y^*$, includes also some, perhaps new, points with just one marker in the row number zero. We will define their images by the factor map in a moment.

Observe that $(\breve Y,\breve S)$, extending the aperiodic system \xttp, is aperiodic, so all the added periodic points are isolated and there is no collision in defining the factor map on them. The only points where continuity of the factor map has to be checked, are the limit points of \sq s of periodic points with growing periods. But for periodic points with large periods $n$ the images are, in the initial $k(n)$ rows, coherent with the standard code defined by the oracle (and composed with the one-block code from $\hat X'$ to $\hat X$). So, in the discussed limit points (regardless of whether these points belong to $\breve Y$ or are newly added), the factor map fully coincides with the standard factor map, due to the fact that $k(n)$ tends to infinity. This means that we have just constructed a \emph{continuous} factor map $\breve\pi^*$ from $(\breve Y^*,\breve S^*)$ onto its image $(\hat X^*, \hat T^*)$ with partial embedding of the sets $\Per^*_n$. In fact $\hat X^*=\hat X$ because the enhanced system already contains all possible points with just one marker in the row number zero (alternatively, we can restrict $\breve Y^*$ to the preimage of $\hat X$). Clearly, $\breve\pi^*$ has the same extension entropy function as $\breve\pi$, because we have added at most countably many preimages to some points. We can now define \ys\ as the preimage by $\breve\pi^*$ of \xt\ (recall that \xt\ is a subsystem of \xtt\ and contains the sets $\Per_n^*$) and this is our desired symbolic extension with partial embedding. 
\qed \end{proof}

The reader uninterested in explanations concerning technicalities in the preceding proofs may pass to Corollary \ref{comb}.

\begin{rem}\label{piecszesc} 
In both Theorems \ref{ensex} and \ref{odwr} we are dealing with similar problems of continuously gluing two factor maps: ``basic'', which is a priori given between some basic system and its basic symbolic extension, with another, ``artificially'' defined between closed sets added to both spaces, consisting of concatenations of pieces occurring in the basic systems and separated by dominant markers (included in the system of all markers). In general, the artificial code agrees with the basic code on regions stretching in some number of initial rows, except on some margins stretching left and right from each dominant marker (resulting from the coding length, usually growing with the row number), as shown on the Figure \ref{gluing}(a) below. 
\begin{figure}
\begin{center}
\includegraphics[width=10cm]{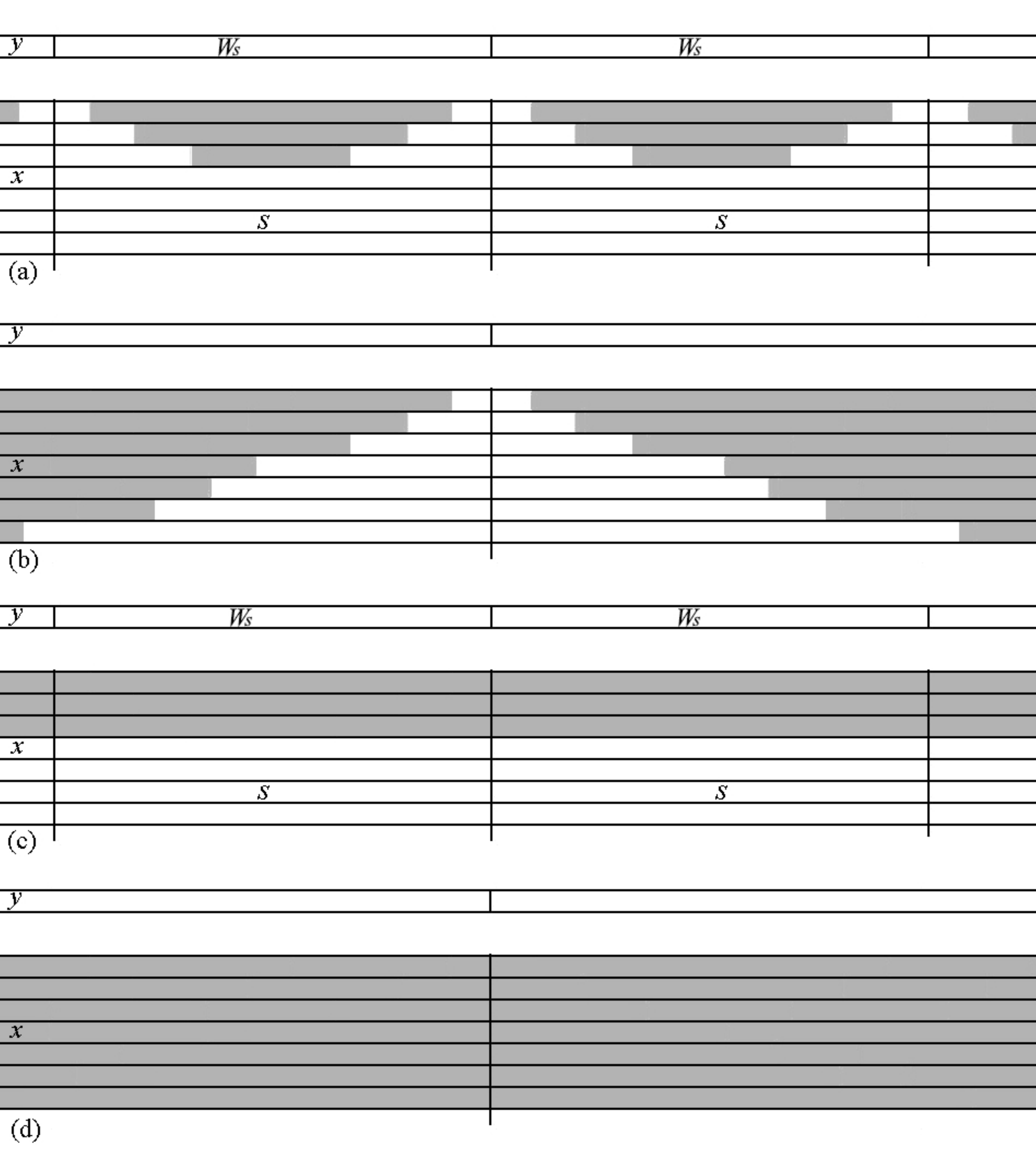}
\caption{\small\label{F:one}The area of a concatenated element where both codes agree.\label{gluing}}
\end{center}
\end{figure}
Some points in the basic systems can be approached by the added concatenated points with increasing distances between the dominant markers, so that these limit points may have at most one dominant marker. 
At points which have such a marker (the juxtaposed elements), the area where the codes agree need not cover the whole array, as shown in gray on Figure \ref{gluing}(b). In Theorem \ref{ensex} the extension is given \emph{a priori} and we cannot assume that it has any special properties. It particular, it cannot 
be standard, because it embeds some periodic orbits, while a standard symbolic extension is always aperiodic. So here the continuous gluing is not guaranteed. Fortunately, in that theorem we can afford the glued extension to be only finitary (and then use a general theorem to replace 
it by a continuous one). 

In Theorem \ref{odwr}, on the other hand, constructing a finitary extension with partial embedding might be insufficient, because while replacing it by a continuous one, we would most likely destroy the embedding. Fortunately, this time we are free to choose the starting symbolic extension as we wish, so we choose a standard one, in which the area where the codes agree stretches from marker to marker, without any margins, as shown on Figure \ref{gluing}(c). Then, in the ``questionable'' limit points with just one dominant marker, this area covers the whole array, and thus the continuous gluing is successful (Figure \ref{gluing}(d)). 
\end{rem}

We can now formulate the combined result which characterizes the entropy functions in symbolic extensions 
with an embedding.

\begin{cor}\label{comb}
Let \xt\ be a zero-dimensional system (or a system with small boundary property) in which we select and fix some finite families $\Per^*_n\subset \Per_n(X,T)$. Let $E$ be a function on $\mtx$. The following conditions are equivalent:
\begin{enumerate}
	\item There exists a symbolic extension $\pi:(Y,S)\to(X,T)$ with partial embedding and with $h^\pi \equiv 
	E$.
	\item $E$ coincides with the restriction to $\mtx$ of a finite affine superenvelope $\hat E$ of the 
	entropy structure of the enhanced system \xtt.
\end{enumerate}
\end{cor}

\begin{rem}\label{measu}
We can now prove that the existence of any, not necessarily measurable, partition $\P$ with $\diam(\P^{[-n,n]})\to 0$ implies the existence of a measurable one (i.e., of a uniform generator). 
For a non-measurable $\P$ the first part of the proof of Theorem \ref{sg} produces a symbolic extension with a perhaps non-measurable equivariant selector from preimages. However, such a selector is measurable on the set of periodic points, because this set is countable (being injectively mapped into the countable set of periodic points of a subshift it is fact finite for every period). Thus Theorem \ref{seraf} applies (with $\Per^*_n=\Per_n(X,T)$) producing a symbolic extension of the enhanced system. Now, we can apply Theorem \ref{odwr} to obtain a symbolic extension of \xt\ with a measurable embedding. Finally, the other implication of Theorem \ref{sg} yields a (measurable) uniform generator. The considerations in the last section (see Theorem \ref{sizea}) imply that the cardinality of the uniform generator can be chosen the same as that of the initial non-measurable partition, because they both lead to symbolic extensions with the same \tl\ entropies. 
\end{rem}

\section{The period tail structure}

We have provided \emph{some} characterization of the extension entropy functions in symbolic extensions with an embedding or partial embedding (the finite sets $\Per_n^* \subset \Per_n(X,T)$ are selected and fixed throughout). But since it refers to the somewhat artificial object, the enhanced system, we are not fully satisfied and the task is now to give a more direct description of these functions. We will give such a characterization,  relying on a new notion, which we call the \emph{period tail structure}, combined with the ``old'' entropy structure.% (jointly they constitute what we call \emph{entropy-and-period tail structure}).

\begin{defn} Let \xt\ be a zero-dimensional \ds\ with selected sets \,$\Per_n^*$ and let $\{\P_k\}$ be some fixed refining \sq\ of clopen partitions. On each set $\Per_n^*$ we define the following \sq\ of functions
$$
\mathfrak P_k(x) = \tfrac1n \log\#\{x'\in\Per_n^*: x' \text{ has the same $\P_k$-name as $x$}\}.
$$
These functions are constant on orbits, so we can naturally associate their values to the ergodic measures
supported by these orbits. On other ergodic measures  we let all these functions equal zero. The functions are then prolonged by integration over the ergodic decomposition\footnote{Further, we will refer to such a prolongation as \emph{harmonic}. Harmonic prolongations are \emph{harmonic functions}, i.e., respect integral averaging, which is in general a stronger condition than just being affine. However, \usc\ affine functions are harmonic.} to the whole of $\mtx$. 
So established \sq\ of functions $\mathfrak P_k:\mtx\to[0,\infty)$ will be called the \emph{period tail structure}. 
\end{defn}
Since the sets $\Per_n^*$ are finite, it is clear that given $n$, the partition $\P_k$ with large enough $k$ separates all points in $\Per_n^*$, hence the functions $\mathfrak P_k$ are zero on them. This implies that the \sq\ $\{\mathfrak P_k\}$ tends (obviously nonincreasingly) to zero pointwise on ergodic measures, and thus also on the entire set $\mtx$.\footnote{In order for a symbolic extension with partial embedding to exist, we must assume that the ``partial'' supremum periodic capacity $\sup_n\frac1n\log(\#\Per^*_n)$ is finite. This enables us to apply the Lebesgue Dominated Convergence Theorem.}
\smallskip

%\begin{defn}The \emph{entropy-and-period tail structure} is defined as the \sq\ of functions $\mathfrak R_k:\mtx\to[0,\infty)$ given by$$\mathfrak R_k= \mathfrak P_k + (h-h_k).$$\end{defn}Note that on each ergodic measure $\mu$ only one of the two ingredients may be positive:if $\mu$ is supported by $\bigcup_n\Per_n^*$ then $\mathfrak R_k(\mu)= \mathfrak P_k(\mu)$,otherwise $\mathfrak R_k(\mu) = (h-h_k)(\mu)$. Clearly, the \sq\ $\{\mathfrak R_k\}$ also tends nonincreasingly to zero. 

In order to formulate how the period tail structure %or the entropy-and-period tail structure 
determines the desired extension entropy functions, we need to recall the concept of a repair function. It is closely related to that of a superenvelope, however, in some sense it eliminates the influence of the limit function. Full details are presented in \cite{Do05} or in the book \cite{Do11}. 

\begin{defn}
Let $\{\theta_k\}$ be a nonincreasing and pointwise converging to zero \sq\ of functions defined on a compact domain. By a \emph{repair function} we mean any nonnegative function $u$ such that $\widetilde{u+\theta_k} - (u+\theta_k)\to 0$ pointwise in $k$ ($\tilde f$ denotes the \usc\ envelope of a function $f$).
\end{defn}
The above convergence can be simplified as $\widetilde{u+\theta_k} \to u$ (nonincreasingly), hence a repair function is always \usc. In the definition we have used a more complicated formula just to highlight the interpretation of a repair function: added to the functions $\theta_k$ it \emph{repairs them} in the sense that the \emph{defects of upper semicontinuity} eventually vanish at every point. An equivalent \emph{repair condition} is: given a point $\mu$ in the domain and $\epsilon>0$ we have 
\begin{equation}\label{reper}
\overline{\lim_i}\,(u+\theta_k)(\mu_i) \le (u+\theta_k)(\mu)+\epsilon,
\end{equation}
for $k$ large enough (the threshold depending on $\mu$ and $\epsilon$), whenever $\mu_i\to \mu$. Observe also that if $u$ repairs $\{\theta_k\}$ then $u\geq \lim_k\tilde{\theta}_k$.\footnote{Equality holds when the \sq\ $\{\theta_k\}$ has so-called order of accumulation $1$. In general, the smallest repair function is obtained by repeating an iterative procedure as many times as the order of accumulation, which is always a countable ordinal. See \cite{Do05} or \cite{Do11} for more details.}
\smallskip

We have the following fairly obvious duality statement, whose easy proof is left to the reader (it can also be found in \cite{Do11}):
\begin{fact}\label{duality}
Let $\{h_k\}$ be a \sq\ of functions on a compact domain, converging pointwise and nondecreasingly to a finite limit function $h$, and such that $h_{k+1} - h_k$ is \usc\ for every $k$. Then $E$ is a superenvelope for $\{h_k\}$ if and only if $E - h$ is a repair function of the \sq\ of tails $\{h-h_k\}$. %In particular, $u+h-h_k$ is \usc\ for each $k$ (we will say, it satisfies the \emph{strong repair condition}).
\end{fact}

We also need  the lemma below:
\begin{lem}\label{uone}
Let $\{\theta_k\}$ be as in Definition 5.2 and let $u_1 = \lim_k\tilde\theta_k$. Then for every \sq\ $x_k\to x$ in the domain we have
$$
\lim_k \theta_k(x_k) \le u_1(x),
$$
and for each $x$ there exists a \sq\ $x_k\to x$ for which the above inequality is an equality.
\end{lem}

\begin{proof}
By monotonicity, $\theta_k(x_k)\le \theta_j(x_k)$ whenever $j\le k$. Thus
$$
\lim_k\theta_k(x_k)\le \lim_k\theta_j(x_k) \le\tilde\theta_j(x),
$$
for each $j$. Taking limit over $j$ we can replace right hand side  by $u_1(x)$.
Further, given $x$, for each $k$ there exists $x_k$ which is $\frac 1k$-close to $x$ ($x_k$ can be equal to $x$) and $\tilde\theta_k(x)\le \theta_k(x_k)+\frac 1k$. Then, for each $k$, 
$$
u_1(x)\le \tilde\theta_k(x) \le \theta_k(x_k)+\tfrac 1k,
$$
hence the reversed inequality follows.
\qed \end{proof} 

We are in a position to give the promised  characterization:

\begin{thm}\label{charact}
The following statements about a nonnegative function $E$ on $\mtx$ are equivalent:
\begin{enumerate}
	\item $E=h^\pi$ for some symbolic extension $\pi:(Y,S)\to(X,T)$ with partial embedding;
	\item $E$ is an affine superenvelope of the entropy structure $\{h_k\}$ and $E\ge h+u_1$, where
	$u_1 = \lim_k \widetilde{\mathfrak P}_k$.
\end{enumerate}
\end{thm}

Before the proof we relate $\mathfrak P_k$ with the tails of the entropy structure of the enhanced system. We view \xt\ in the array representation generated by the \sq\ of clopen partitions $\{\P_k\}$ and we create the enhanced system \xtt\ using this representation. In order to differentiate between the entropy functions and structures on \xt\ and on \xtt, in the latter case we will write $\hat h$ and $\hat h_k$, respectively.

\begin{lem}\label{mme}
If $\mu$ is periodic supported by $\Per_n^*$ and $k\ge 1$ is such that $\mathfrak{P}_k(\mu)>0$, then there exists an \im\ $\mu^{(k)}$ supported by  $\mathbf{S}_n$, such that $\mathfrak P_k(\mu) = (\hat h-\hat h_k)(\mu^{(k)})$.  For other ergodic measures $\mu$ (with $\mathfrak{P}_k(\mu)=0$)  we let $\mu^{(k)}=\mu$ and finally we prolong the assignment $\mu\mapsto \mu^{(k)}$ harmonically to a map $\mtx\to\mtxt$. Then
$$
\mathfrak P_k(\mu) \leq (\hat h-\hat h_k)(\mu^{(k)}). 
$$
The weak-star distance between $\mu$ and $\mu^{(k)}$ does not exceed some $\delta_k$   which tends to zero with $k$.
\end{lem}
\begin{proof}
Namely, when $\mu$ is periodic supported by $\Per_n^*$, then  $\mu^{(k)}$ is the measure of maximal entropy among those in $\mathcal{M}_{\hat T} (\mathbf{S}_n)$ which have the same (periodic) projection on the top $k$ rows factor $X_k$ as $\mu$. The verification is fairly obvious and we skip it. 
\qed \end{proof}

%We extend harmonically the map $\mu\mapsto \mu^{(k)}$ on the set of invariant  measures supported by $\bigcup_n\Per_n^*$. As the  weak-star distance may be chosen to be convex the measures $\mu$ and $\mu^{(k)}$ are again $\delta_k$-close. 

\begin{lem}\label{mmf}
Conversely, if $\mu$ is supported by $\mathbf S_n$, then there exists a convex combination $\check\mu\in\mtx$ of the periodic measures supported by $\Per^*_n$, such that, for every $k$, 
$$
(\hat h-\hat h_k)(\mu)\le\mathfrak P_k(\check\mu).
$$
The weak-star distance between $\mu$ and $\check\mu$ does not exceed some $\gamma_n$ which tends to zero with $n$.
\end{lem}

\begin{proof}
Namely, $\check\mu$ is the convex combination $\sum_{s\in\mathcal S_n}\mu([s])\cdot\mu_s$, where $\mu_s$ is the periodic measure supported by the orbit of the array obtained by concatenating repetitions of the strip $s$, and $[s]$ denotes the cylinder associated to the strip $s$. Again, the fairly obvious verification will be skipped.
\qed \end{proof}

\begin{proof}[Proof of Theorem \ref{charact}]
By Theorem \ref{env} and the duality Fact \ref{duality}, the extension entropy function $h^\pi$ in a symbolic extension with partial embedding  equals $h+u$ where $u$ is an affine repair function of the entropy tail structure $\{h-h_k\}$. The implication (1)$\implies$(2) will be proved once we show that $u=h^\pi-h$ is larger than or equal to $u_1$.

By Corollary \ref{comb}, the function $h^\pi$ in a symbolic extension with partial embedding coincides with the restriction to $\mtx$ of an affine superenvelope $\hat E$ of the entropy structure of the enhanced system. The duality Fact \ref{duality} implies that $\hat E=\hat h + \hat u$ where $\hat u$ is an affine repair function of the \sq\ of tails $\{\hat h-\hat h_k\}$. Since $\hat h|_{\mtx} = h$, we have $\hat u|_{\mtx}=u$.  By the second part of Lemma \ref{uone}, for any $\mu\in \mtx$ there exists a \sq\ $(\mu_k)$ in $\mtx$, converging to $\mu$, and such that 
$$
u_1(\mu)=\lim_k \mathfrak P_k(\mu_k)\le \lim_k(\hat h -\hat h_k)(\mu_k^{(k)}),
$$
where the latter inequality refers to measures $\mu_k^{(k)}$ produced for each $\mu_k$ by Lemma~\ref{mme}. Since these measures also converge to $\mu$, the first part of Lemma \ref{uone} yields that the right hand side above does not exceed
$$
\lim_k \widetilde{(\hat h - \hat h_k)}(\mu)\le \lim_k \widetilde{(\hat u+\hat h - \hat h_k)}(\mu)=\hat u(\mu) = u(\mu).
$$
This completes the proof of (1)$\implies$(2).
\smallskip

The implication (2)$\implies$(1) will be proved using two further lemmas.
\begin{lem}
Let $\{\hat\theta_k\}$ be a nondecreasing and pointwise converging to zero \sq\ of functions on a compact domain $\hat \M$. Let $\theta_k = \hat\theta_k|_{\M},\ \theta'_k = \hat\theta_k \mathbf 1_{\hat \M\setminus \M}$, where $\M$ is a compact subset of \, $\hat \M$. If \, $u$ repairs $\{\theta_k\}$ on $\M$, $u'$ repairs $\{\theta_k'\}$ on $\hat \M$ and $u\ge u'|_\M$, then $\hat u$ defined as $u$ on $\M$ and $u'$ on $\hat \M\setminus \M$ repairs $\{\hat\theta_k\}$. 
\end{lem}
\begin{proof}
We will verify the repair condition \eqref{reper} for $\hat u$ and $\{\hat\theta_k\}$. Let $\mu_i\to\mu$ in $\hat \M$. If $\mu\in \hat \M\setminus \M$ then $\mu_i\in \hat \M\setminus \M$ for large $i$ and the repair condition follows from the properties of $u'$. If both $\mu\in \M$ and $\mu_i$ are in $\M$, the condition follows from the properties of $u$. In the remaining case $\mu\in \M, \mu_i\in \hat \M\setminus \M$ we have, for large~$k$,
$$
\overline{\lim_i}\,(\hat u+\hat\theta_k)(\mu_i)= \overline{\lim_i}\,(u'+\theta'_k)(\mu_i) \le (u'+\theta'_k)(\mu)+\epsilon \le u(\mu)+\epsilon\le
(\hat u+\hat \theta_k)(\mu)+\epsilon. 
$$
\qed \end{proof}

\begin{lem}
On the (compact) set $\hat \M=\mtx\cup\bigcup_n\mtsn$ (with $\M=\mtx$) 
define $u'$ as $u_1=\lim_k \widetilde{\mathfrak P}_k$ on $\M$ and zero otherwise. Then $u'$ 
repairs the \sq\ of entropy tails $\theta'_k=(\hat h-\hat h_k)\mathbf 1_{\hat \M\setminus \M}$ . 
\end{lem}

\begin{proof}
We need to verify the repair condition \eqref{reper} for $\mu_i\to\mu$ in $\hat \M$. If $\mu\notin\mtx$ then both $\mu$ and $\mu_i$ with large $i$ belong to $\mtsn$ for some $n$ (the sets $\mathbf S_n\setminus X$ are open). Since $\mathbf S_n$ is expansive, $\{\hat h-\hat h_k\}$ is upper semicontinuous on $\mtsn$, hence $u'=0$ verifies the repair condition. If both $\mu$ and all $\mu_i\in \M$ then, since $u_1$ is \usc, we have
$$
\overline{\lim_i}\,(u'+\theta'_k)(\mu_i)= \overline{\lim_i}\, u_1(\mu_i)\le u_1(\mu)=(u'+\theta'_k)(\mu).
$$
In the case $\mu\in \M$ and $\mu_i\in \hat \M\setminus \M$ we have $\mu_i\in \mathcal M_{\hat T}(\mathbf S_{n_i})$, where the indices $n_i$ grow to infinity. Then we invoke the measures $\check\mu_i$ provided for $\mu_i$ by Lemma~\ref{mmf}. They are supported by $X$ and also converge to $\mu$, thus
$$
\overline{\lim_i}\, (u'+\theta'_k)(\mu_i) = \overline{\lim_i}\, \theta'_k(\mu_i)
\le \overline{\lim_i}\, \mathfrak P_k(\check\mu_i)\le \widetilde{\mathfrak P}_k(\mu)\le u_1(\mu)+\epsilon = 
(u'+\theta'_k)(\mu)+\epsilon, 
$$ 
for $k$ large enough.
\qed \end{proof}

Now let $E\ge h+u_1$ be an affine superenvelope of the entropy structure of \xt. As already noticed, $E=h+u$ where $u$ is a repair function for the entropy tails on $X$, larger than or equal to $u_1$. Combining the above two lemmas, we get that then $\hat u$, defined as $u$ on $\mtx$ and zero on $\hat \M\setminus \M$, repairs the entropy tails $\{\hat h-\hat h_k\}$ restricted to $\hat\M=\mtx\cup\bigcup_n\mtsn$. Since $u$ is affine, being \usc, it is harmonic on $\mtx$. Since $\mtx$ is a \emph{face} of $\mtxt$ (i.e. a sub-simplex spanned by a subset of extreme points), this easily implies that $u$ coincides with the restriction to $\mtx$ of the harmonic prolongation ${\hat u}^{\mathsf{har}}$ of $\hat u$ onto the entire set $\mtxt$. Note that if we prolong $\hat u$ to $\hat u_0$ by assigning zero to all measures in $\mtxt$ on which $\hat u$ is still undefined, then we obtain a convex and \usc\ function, which has the same harmonic prolongation as $\hat u$.\footnote{For the harmonic prolongation only the values at extreme points matter and the domain of $\hat u$ contains all ergodic measures of \xtt.} By \cite[Fact A.2.20 ]{Do11}, ${\hat u}^{\mathsf{har}} = {\hat u_0}^{\mathsf{har}}$ is upper semicontinuous. Also, ${\hat u}^{\mathsf{har}}$ coincides with $\hat u$ on the set $\hat\M$ which contains the closure of all ergodic measures of \xtt. Since $\hat u$ repairs $\{\hat h-\hat h_k\}$ on $\hat \M$ it follows from \cite[Lemma 8.2.14]{Do11} that ${\hat u}^{\mathsf{har}}$ is an affine repair function for $\{\hat h-\hat h_k\}$ on the entire set $\mtxt$. As we have already argued, $\hat E = h+ {\hat u}^{\mathsf{har}}$ represents an extension entropy function in a symbolic extension of the enhanced system, and its restriction to $\mtx$ thus represents an extension entropy function in a symbolic extension of \xt\ with a partial embedding. Since, as noticed earlier, ${\hat u}^{\mathsf{har}}|_{\mtx}=u$, the above restriction equals $h+u=E$, 
which establishes the implication (2)$\implies$(1).  
\qed \end{proof}

\section{Final remarks and examples}

In this section we state some consequences of Theorem \ref{charact} combined with the general theory of symbolic extensions and we provide examples showing that the characterization in Theorem \ref{charact} cannot be essentially simplified, i.e., expressed in a more direct way. For simplicity, we assume that 
$\Per^*_n = \Per_n(X,T)$ for each $n$ (i.e., we will focus on symbolic extension with a \emph{complete} embedding), although analogous notions, statements and examples can be produced for partial embedding as well. We introduce the following notions:

\begin{defn}{\color{white}{.}}
\begin{itemize}
	\item Let $\hemb$ denote the function on $\mtx$ defined as the pointwise infimum of $h^\pi$  
	over all symbolic extensions with an embedding,
	\item let $\hemb(X,T)$ be the infimum of $\htop(Y,S)$ over all symbolic extensions with an embedding,
	\item let $\mathfrak P^*(X,T) =\lim_k \sup_{x\in X}\mathfrak P_k(x) = \lim_k \sup_{\mu\in\mtx}\mathfrak P_k(\mu)$.
\end{itemize}
\end{defn}

Since $\mathfrak P_k$ and $\widetilde{\mathfrak P}_k$ have the same supremum, and $\{\widetilde{\mathfrak P}_k\}$ is a nonincreasing sequence of \usc\ functions we may swap the supremum and the limit (see e.g. \cite[Proposition 2.4]{BD05}), so that we obtain
$$
\mathfrak P^*(X,T) = \lim_k\sup_{\mu\in\mtx}\widetilde{\mathfrak P}_k(\mu) = \sup_{\mu\in\mtx}\lim_k\widetilde{\mathfrak P}_k(\mu) = \sup_{\mu\in\mtx} u_1(\mu).
$$
Thus our notation $\mathfrak P^*(X,T)$ stands in analogy to the notion of tail entropy\footnote{Tail entropy is known mainly under the confusing name of ``\tl\ conditional entropy''. It turns out (see \cite{Do05}) that $h^*(X,T)$ is equal to $\sup u^{\mathcal H}_1$, where $u^{\mathcal H}_1$ is computed for the tails of the entropy structure as $u^{\mathcal H}_1=\lim_k(\widetilde{h-h_k})$.} $h^*(X,T)$ defined by Misiurewicz and we will call it the \emph{tail period capacity}. A corresponding \emph{variational principle for the embedding entropy} also holds:
$$
\hemb(X,T) = \sup_{\mu\in\mtx}\hemb(\mu).
$$
The proof follows the same scheme as the proof of the ``symbolic extension entropy variational principle''
(see \cite{BD05} or the book \cite{Do11}). Since we will not use it, we refrain from providing a detailed proof (which requires introducing a few more notions). Instead, we prove some more useful estimates:

\begin{fact}\label{hemb}One has the following inequalities:
\begin{gather*}
\max\{\hsex,h+u_1\}\le \hemb\le \hsex+u_1,\\
\max\{\hsex(X,T),\mathfrak P^*(X,T)\}\le\hemb(X,T)\le\hsex(X,T)+\mathfrak P^*(X,T).
\end{gather*}
\end{fact} 
\begin{proof} The left hand side inequalities follow directly from Theorem \ref{charact}. Since each function $\widetilde{\mathfrak P}_k$ \usc\ and concave (the \usc\ envelope of an affine function is always concave) we can write $u_1$ as a pointwise infimum of \usc\ affine functions $G$. Whenever $E$ is an affine superenvelope of $\{h_k\}$, so are the sums $E+G$, and clearly they are larger than or equal to $h+u_1$. Taking infimum over $E$ and $G$ we get the first right hand side inequality. The second one is now trivial since the supremum of a sum of two functions does not exceed the sum of their suprema.
\qed \end{proof}

And this is actually all that can be said in full generality. The inequalities cannot be improved in general. We will provide examples in which the function $\hemb$ actually equals $h^\pi$ in a suitable extension and realizes the following cases: ($f\lneqq g$ means $f\le g$ and $f(x)<g(x)$ at some point)
\begin{itemize}
	\item In Example \ref{example2}: 
	\begin{gather*}
	\max\{\hsex, h+u_1\}\lneqq\hemb = \hsex+u_1,\\ 
	\max\{\hsex(X,T), \mathfrak P^*(X,T)\}<\hemb(X,T) =  
	\hsex(X,T)+\mathfrak P^*(X,T);
	\end{gather*} 
	\item In Example \ref{example3}:
	\begin{gather*}
	\max\{\hsex, h+u_1\}=\hemb\lneqq\hsex+u_1,\\ 
	\max\{\hsex(X,T), \mathfrak P^*(X,T)\}=\hemb(X,T) <  
	\hsex(X,T)+\mathfrak P^*(X,T);
	\end{gather*} 
	\item We will show that in the old Example \ref{example1}, $\hemb$ is not equal to $h$ plus the minimal 
	repair function of any tail structure combining $\{h-h_k\}$ and $\mathfrak P_k$, in particular of 
	$\{h-h_k+\mathfrak P_k\}$ or of $\{\max\{h-h_k,\mathfrak P_k\}\}$. In fact, in this example we have 
	$h\equiv0$ while $u_1$ alone is not a repair function of the tail structure $\{\mathfrak P_k\}$. 
\end{itemize}
All this means that the way in which the act of embedding periodic points in a symbolic extension affects the function $\hsex$ (to become $\hemb$) is complicated and difficult to predict. It seems that the characterization in Theorem \ref{charact} is indeed the best possible.

\begin{exam}\label{example2}
Let $Y_0$ be a uniquely ergodic subshift of positive entropy $h_0$, and denote by $\mu$ its unique \im. 
Fix some $y_0\in Y_0$. By unique ergodicity, $y_0$ is generic for $\mu$. Let $y_p$ be the periodic \sq\ obtained by repeating the initial block of length $p$ of $y_0$ (we can assume that the minimal period of $y_p$ is $p$). Let $Y$ be the smallest subshift containing $Y_0$ and each $y_p$. The ergodic measures of $Y$
are $\mu$ and a \sq\ of periodic measures converging to $\mu$. Now let $X$ be the array system consisting of arrays $x$ of the following form: $x$ has at most two nonzero rows: first of them, number $m$, contains an element $y$ of $Y$, the second one, present only if $y$ is periodic with minimal period $p$ (i.e., when $y$ is in the orbit of $y_p$), number $m+p$, contains another (arbitrary, say over $\{0,1\}$) periodic \sq\ with minimal period $p$. The structure of ergodic measures of $X$ is as follows: there are ``clusters'' of periodic measures $\mu_{m,p,i}$ of period $p$, supported by arrays with two nonzero rows, $m$ and $m+p$ (the index $i$ enumerates all periodic patterns with minimal period $p$). As $p$ grows, these measures converge (regardless of $i$) to measures $\mu_m$ supported by arrays with only one nonzero row (number $m$) which contains an element of $Y_0$. As $m$ grows, all measures $\mu_{m,p,i}$ and $\mu_m$ tend to $\mu_0$---the pointmass at the fixpoint (the array of zeros). Note that the set of ergodic measures is closed, so in order to study affine superenvelopes of \sq s of affine functions on $\mtx$, it suffices to study arbitrary superenvelopes on the set of ergodic measures (the situation is shown on Figure \ref{fexample2}).
\begin{figure}[ht]
\begin{center}
\includegraphics[width=7cm,height=6cm]{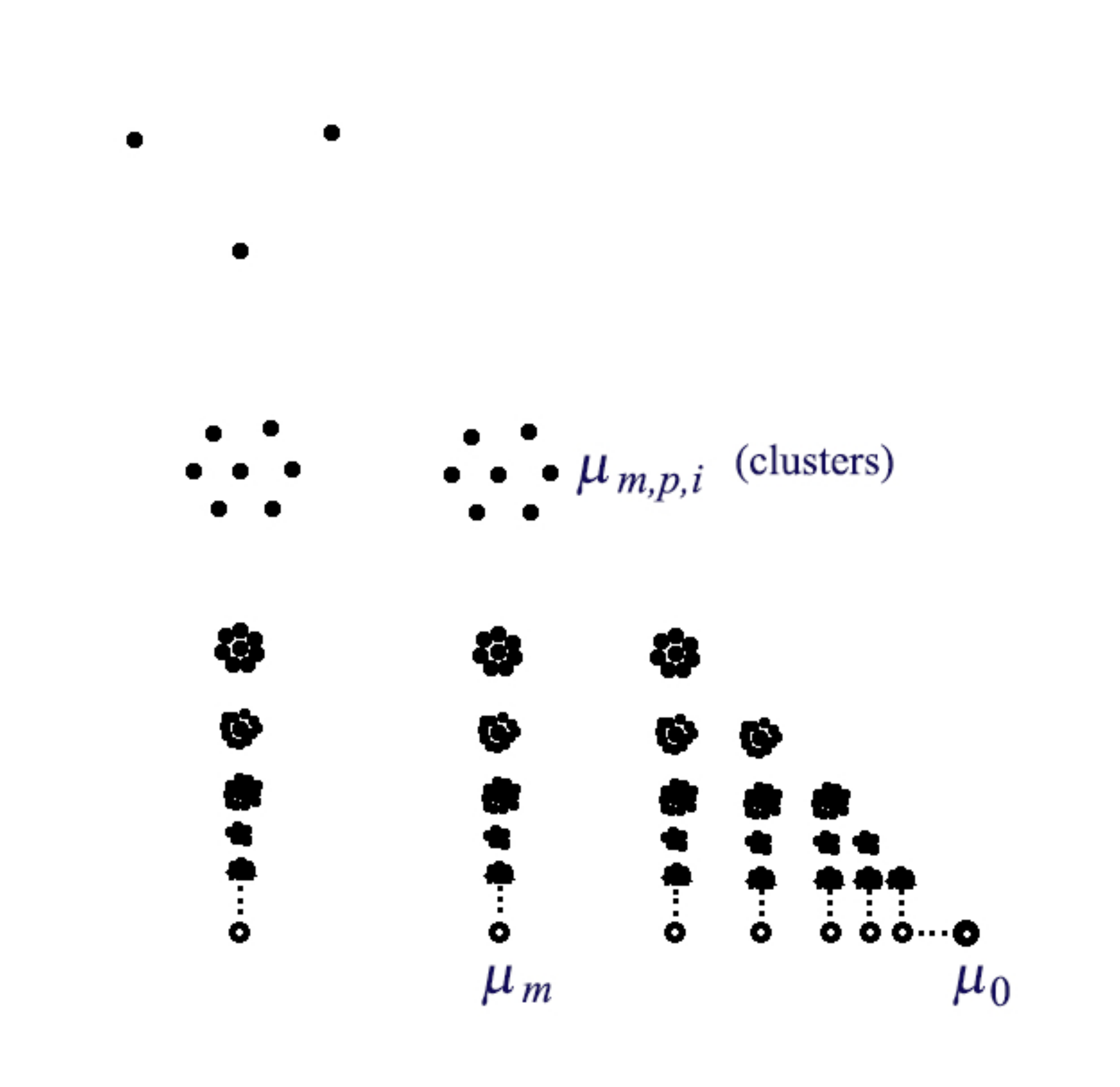}
\caption{\small\label{f9} The set of ergodic measures in the example.\label{fexample2}}
\end{center}
\end{figure}

All measures $\mu_m$ have entropy $h_0$, and $(h-h_k)(\mu_m)=0$ whenever $k\ge m$ and $h_0$ otherwise. This easily implies that the minimal repair function $u$ of $\{h-h_k\}$ equals zero except at $\mu_0$ where it equals $h_0$, and $\hsex = h+u$ equals $h_0$ on the measures $\mu_m$, $\mu_0$ and clearly zero on each $\mu_{m,p,i}$. 

On the other hand, the functions $\mathfrak P_k$ nearly equal $\log 2$ at each $\mu_{m,p,i}$ whenever $k\le m+p$ (nearly, because in row $m+p$ we must not use blocks of length $p$ whose repetitions have minimal period smaller than $p$), which easily implies that $u_1 = \lim_k\widetilde{\mathfrak P}_k$ equals $\log 2$ at each measure $\mu_m$, and, by upper semicontinuity also at $\mu_0$ (and zero at other ergodic measures). In particular $\mathfrak P^*(X,T)=\log 2$.

Now let $u$ be a repair function of $\{h-h_k\}$ larger than or equal to $u_1$, i.e., larger than or equal to $\log 2$ on the measures $\mu_m$ and $\mu_0$. Then $(u+h-h_k)(\mu_m)\ge \log 2+h_0$ for $m>k$ implying,
by the repair condition, that $u(\mu_0) \ge \log 2 + h_0$. Thus
$$
\hemb(\mu_0) = h_0 +\log 2= \hsex(\mu_0)+u_1(\mu_0)>\max\{\hsex(\mu_0),h(\mu_0)+u_1(\mu_0)\}.
$$  
At other ergodic measures we have $\hsex+u_1=\max\{\hsex,h+u_1\}$ (we skip the easy verification), so $\hemb$ equals both sides. Nonetheless, the topological notions depend only on $\mu_0$ and we have
$$
\hemb(X,T) = h_0 +\log 2= \hsex(X,T)+\mathfrak P^*(X,T)>\max\{\hsex(X,T),\mathfrak P^*(X,T)\}.
$$ 
\end{exam}

\begin{exam}\label{example3}
This example is a simplification of the preceding one. The system $X$ consists of arrays with only one nontrivial row, number $m$, which contains either an element of $Y_0$ or an arbitrary periodic \sq\ over $\{0,1\}$, of minimal period $m$. There are now ergodic measures $\mu_m$ of entropy $h_0$ tending to $\mu_0$ and clusters of periodic measures $\mu_{m,i}$ also tending (with increasing $m$, regardless of $i$) to $\mu_0$. The entropy structure is basically the same as in the preceding example, in particular $\hsex(\mu_0)=h_0$. However, this time $u_1$ equals $\log 2$ only at $\mu_0$ (and zero otherwise). The function $u$ equal to $\max\{h_0,\log 2\}$ at $\mu_0$ (and zero at other ergodic measures) is larger than or equal to $u_1$, and repairs the tails of the entropy structure, thus
$$
\hemb(\mu_0) = \max\{\hsex(\mu_0),h(\mu_0)+u_1(\mu_0)\}< \hsex(\mu_0)+u_1(\mu_0).
$$
At other ergodic measures $\max\{\hsex,h+u_1\}= \hsex+u_1$ (because $u_1=0$), so that $\hemb$ equals both sides. Nonetheless, we have
$$
\hemb(X,T) = \max\{h_0,\log 2\}= \max\{\hsex(X,T),\mathfrak P^*(X,T)\}<\hsex(X,T)+\mathfrak P^*(X,T).
$$ 
\end{exam}

\smallskip\noindent
\emph{Example \ref{example1}, continuation}. Recall that there are measures $\mu_{m,i,j,l}$ converging
with $j$ (and regardless of $l$) to measures $\mu_{m,i}$, which in turn converge with $m$ (regardless of $i$) to $\mu_0$. In the second (more sophisticated) symbolic extension with an embedding, $h^\pi$ equals zero at each $\mu_{m,i,j,l}$, and $\log 2$ on each $\mu_{m,i}$ and on $\mu_0$. It is elementary to see that the period tail structure in this example is as follows: $\mathfrak P_{k}(\mu_{m,i,j,l}) = \log 2$ for $k<j$ and $0$ for $k\ge j$. Similarly, $\mathfrak P_{m}(\mu_{m,i}) = \log 2$ for $k<m$, and zero otherwise. All these functions are zero at $\mu_0$. This period tail structure has the same form as the ``pick up sticks game 3'' on page 232 in \cite{Do11} except that single points must be replaced by clusters of measures having common indices $k$ and $j$. As explained in the book, $u_1$ equals $0$ at all measures $\mu_{k,i,j,l}$ and $\log 2$ at all measures $\mu_{k,i}$ and $\mu_0$ (hence it matches the above $h^\pi$), and this function \emph{does not repair} the period tail structure. The order of accumulation of this structure is 2 and the smallest repair function is $u_2$ which assumes the value $2\log 2$ at $\mu_0$ (see figure on page 234 in \cite{Do11}). This function actually matches $h^\pi$ in the first (more obvious) symbolic extension with an embedding. Because in this example $h=h_k=0$ for all $k$, $h+u_1$ is not a repair function of neither $\{h-h_k+\mathfrak P_k\}$ nor $\{\max\{h-h_k,\mathfrak P_k\}\}$, nor any other (reasonable) combination of these two tail structures.
\bigskip

We investigate now the optimal cardinality of a uniform generator. 
Recall that it equals the optimal size of an alphabet $\Lambda$ of a symbolic extension $(Y,S)$ with an embedding. As we have already noticed, in the aperiodic case this cardinality equals the least integer $l$ with $\log l>\hsex(X,T)$, that is $\#\Lambda = \lfloor2^{\hsex(X,T)}\rfloor+1$. In presence of periodic points we have:
 
\begin{thm}\label{sizea}
Let \xt\ be a topological dynamical system admitting a uniform generator. Then the optimal cardinality of a uniform generator equals
$$
\#\Lambda=\lfloor2^{\max\{\psup (X,T),\hemb(X,T)\}}\rfloor+1,
$$
which does not exceed $\lfloor2^{\max\{\psup(X,T),\hsex(X,T)+\mathfrak P^*(X,T)\}}\rfloor+1$.
\end{thm}
  
\begin{proof}
By Remark \ref{kr}, it is enough to see that 
$$
\hsex(\hat X,\hat T)\le\max\{\psup(X,T),\hemb(X,T)\}.
$$ 
Indeed if $\pi:(Y,S)\to(X,T)$ is a symbolic extension with an embedding then the finitary symbolic extension $\hat\pi:(\hat Y, \hat S)\to(\hat X,\hat T)$ obtained in Lemma \ref{ensex} has entropy equal to $\max\{\sup_n\htop(\mathbf S_n), \htop(Y,S)\}=\max\{\psup(X,T),\htop(Y,S)\}$. Then one concludes by Serafin's Theorem (or Theorem \ref{seraf}).
\qed \end{proof}

\begin{rem}
When $(X,T)$ is \emph{asymptotically expansive}, i.e., it is asymptotically $h$-expansive and $\mathfrak P^*(T)=0$, then we may choose a uniform generator of cardinality
$$
\#\Lambda=\lfloor2^{\max\{\psup(X,T),\htop(X,T)\}}\rfloor+1.
$$ 
Thus we recover the estimate obtained in \cite{Bu16}. 
\end{rem}

We will  say now a few concluding words about the applications for smooth dynamical systems. Recall that $C^\infty$ maps on compact manifolds are  asymptotically $h$-expansive \cite{Buz97}, thus the entropy function $h$ is a superenvelope of the entropy structure. The first author proved in a recent paper \cite{Bu16'}   that for a $C^\infty$ surface diffeomorphism $f:M\rightarrow M$ and for any $\delta>0$ we have $$\mathfrak P^*(M,f)=0,$$  when  the subset $\Per^*=\Per^*_\delta$ of selected periodic points  is given  by hyperbolic saddles with Lyapunov exponents $\delta$-away from zero.  In this case we get therefore as a consequence of Theorem \ref{charact}:

\begin{cor}
For any $C^\infty$ surface diffeomorphism $(M,f)$ and for any $\delta>0$ there exists a symbolic extension $\pi:(Y,S)\to(M,f)$ with partial embedding with respect to $\Per^*_\delta$  such that 
$h^\pi = h$.
\end{cor}

\end{document}